\documentclass[11pt]{amsart}

\usepackage[margin=1.25in]{geometry}
\usepackage[T1]{fontenc}
\usepackage{lmodern}
\usepackage{microtype}

\usepackage{amsmath,amssymb,amsthm,mathtools}
\numberwithin{equation}{section}
\allowdisplaybreaks

\usepackage{enumitem}
\usepackage[hidelinks]{hyperref}

\usepackage{tikz}
\usetikzlibrary{arrows.meta,positioning,calc}

\newtheorem{theorem}{Theorem}[section]
\newtheorem{proposition}[theorem]{Proposition}
\newtheorem{lemma}[theorem]{Lemma}
\newtheorem{corollary}[theorem]{Corollary}

\theoremstyle{definition}
\newtheorem{definition}[theorem]{Definition}

\theoremstyle{remark}
\newtheorem{remark}[theorem]{Remark}

\title[Sampling and interpolation in mixed-norm Fock spaces]
{Sampling and Interpolation in Gaussian Mixed-Norm Fock Spaces}

\author[X. Fang]{Xiang Fang}
\address{Department of Applied Mathematics\\
National Yang Ming Chiao Tung University\\
Hsinchu 30010, Taiwan}
\email{xfang@nycu.edu.tw}

\author[P. T. Tien]{Pham Trong Tien}
\address{VNU University of Science\\
Vietnam National University, Hanoi\\
Vietnam}
\email{phamtien@vnu.edu.vn}

\date{}



\hypersetup{
  pdftitle={Sampling and Interpolation in Gaussian Mixed-Norm Fock Spaces},
  pdfauthor={Xiang Fang and Pham Trong Tien},
  pdfsubject={Sampling and interpolation in radial--angular Gaussian
    mixed-norm Fock spaces},
  pdfkeywords={Fock spaces, radial--angular mixed norms, sampling,
    interpolation, Beurling density, quasi-Banach spaces, localized matrices}
}

\begin{document}

\begin{abstract}
We establish sharp sampling and interpolation criteria for the
radial--angular Gaussian mixed-norm Fock spaces
$\mathcal F_\alpha^{p,q}$, where $\alpha>0$ and
$0<p,q\le\infty$. The natural point-evaluation scale leads to the
normalized restriction
\[
R_{\alpha,\Lambda}^{p,q}f
=
\bigl(
f(\lambda)e^{-\frac{\alpha}{2}|\lambda|^2}
(1+|\lambda|)^{\frac{1}{q}-\frac{1}{p}}
\bigr)_{\lambda\in\Lambda}.
\]
With this normalization, the critical density is the classical value
$\frac{\alpha}{\pi}$ throughout the Banach and quasi-Banach ranges.

\smallskip
\noindent
For an arbitrary locally finite set $\Lambda$, sampling is equivalent to
\[
D_{\mathrm{sep}}^{-}(\Lambda)>\frac{\alpha}{\pi},
\]
together with relative separation when $p<\infty$; no
relative-separation condition is required when $p=\infty$. Here
$D_{\mathrm{sep}}^{-}$ is the supremum of the lower Beurling densities
over separated subsets of $\Lambda$. Interpolation is equivalent to
separation and
\[
D^{+}(\Lambda)<\frac{\alpha}{\pi}.
\]
The same criterion characterizes interpolation for the little endpoint
$f_\alpha^{p,\infty}$, and in both settings the normalized restriction
map admits a bounded linear right inverse.

\smallskip
\noindent
The sufficiency arguments are based on rapidly localized Hilbert dual
and Lagrange atoms together with a weighted localized synthesis theorem
on the full mixed-norm scale. For necessity, we prove that lower
stability of a rapidly localized matrix on a weighted annular mixed
sequence space implies lower stability on $\ell^2$. The proof first
passes to $\ell^\infty$ by translated polynomial cutoffs and a
commutator estimate, and then uses the $p$-independence theorem for the
Sjöstrand class. Applied through a fixed Hilbert sampling lattice, this
reduces the strict density conditions to the classical Hilbert Fock
sampling and interpolation theorems.
\end{abstract}

\maketitle

\section{Introduction and main results}
\label{S:Introduction}

\subsection{Classical density and the radial--angular scale}
\label{SS:IntroductionClassicalMixedScale}

Sampling and interpolation are two fundamental discretization problems for
analytic function spaces.  Sampling asks when point values determine every
function stably, while interpolation asks when arbitrary admissible data can
be prescribed with controlled norm.  For the classical Bargmann--Fock
space, the sharp answer is given by the density theorems of Seip,
Wallst\'en, and Lyubarskii: sampling requires strict lower density above the
critical value, whereas interpolation requires strict upper density below
it; see \cite{Lyubarskii1992,Seip1992,SeipWallsten1992,Zhu2012}.  With Gaussian parameter $\alpha>0$, the critical
density is
\[
\frac{\alpha}{\pi}.
\]

\smallskip
\noindent
The purpose of this paper is to determine the corresponding discrete
geometry for the radial--angular mixed-norm Fock spaces.  
For an entire function $f$ and $r\ge 0$, define
\[
M_p(f,r)
:=
\begin{cases}
\displaystyle
\left(
\frac{1}{2\pi}
\int_0^{2\pi}
|f(re^{i\theta})|^p\,d\theta
\right)^{\frac{1}{p}},
& 0<p<\infty,\\[3mm]
\displaystyle
\sup_{0\le\theta<2\pi}|f(re^{i\theta})|,
& p=\infty.
\end{cases}
\]
For $\alpha>0$ and $0<p,q\le\infty$, the mixed-norm Fock space
$\mathcal F_\alpha^{p,q}$ consists of all entire functions $f$ such that
\[
\|f\|_{\mathcal F_\alpha^{p,q}}<\infty,
\]
where, for $0<q<\infty$,
\begin{equation*}
\label{eq:intro-mixed-norm}
\|f\|_{\mathcal F_\alpha^{p,q}}
:=
\left(
\alpha q
\int_0^\infty
M_p(f,r)^q
e^{-\frac{\alpha q}{2}r^2}
r\,dr
\right)^{\frac{1}{q}},
\end{equation*}
and, for $q=\infty$,
\[
\|f\|_{\mathcal F_\alpha^{p,\infty}}
:=
\sup_{r\ge0}
M_p(f,r)e^{-\frac{\alpha}{2}r^2}.
\]
At the outer endpoint $q=\infty$, we also use the little mixed-norm
Fock space
\[
f_\alpha^{p,\infty}
:=
\left\{
f\in\mathcal F_\alpha^{p,\infty}:
M_p(f,r)e^{-\frac{\alpha}{2}r^2}\longrightarrow0
\ \text{as } r\to\infty
\right\}.
\]

\smallskip
\noindent
The two exponents encode different geometric features.  The inner exponent
$p$ measures angular integrability on each circle, whereas the outer
exponent $q$ measures radial accumulation of the resulting circle means. When $p=q$, we write
\[
\mathcal F_\alpha^p:=\mathcal F_\alpha^{p,p},
\]
which is the classical one-parameter Fock space.  For $p\ne q$, angular concentration and radial accumulation are
measured independently, and the norm is no longer described by a single
planar Lebesgue exponent.

\smallskip
\noindent
This radial--angular structure also determines the natural discrete model.
Given a locally finite set $\Lambda\subset\mathbb C$, let
\[
A_k:=\{z\in\mathbb C:k\le |z|<k+1\},
\qquad
\Lambda_k:=\Lambda\cap A_k,
\qquad k\ge0.
\]
For a sequence $a=(a_\lambda)_{\lambda\in\Lambda}$, set
\[
\|a|_{\Lambda_k}\|_{\ell^p(\Lambda_k)}
:=
\begin{cases}
\displaystyle
\left(\sum_{\lambda\in\Lambda_k}|a_\lambda|^p\right)^{\frac{1}{p}},
& 0<p<\infty,\\[3mm]
\displaystyle
\sup_{\lambda\in\Lambda_k}|a_\lambda|,
& p=\infty,
\end{cases}
\]
with the convention that the norm of an empty block is zero.
For $0<p,q\le\infty$, the annular mixed sequence space
$\ell^{p,q}(\Lambda)$ consists of all sequences
$a=(a_\lambda)_{\lambda\in\Lambda}$ such that
\[
\|a\|_{\ell^{p,q}(\Lambda)}<\infty,
\]
where
\begin{equation}
\label{eq:intro-annular-sequence}
\|a\|_{\ell^{p,q}(\Lambda)}
:=
\begin{cases}
\displaystyle
\left(
\sum_{k=0}^{\infty}
\|a|_{\Lambda_k}\|_{\ell^p(\Lambda_k)}^q
\right)^{\frac{1}{q}},
& 0<q<\infty,\\[4mm]
\displaystyle
\sup_{k\ge0}
\|a|_{\Lambda_k}\|_{\ell^p(\Lambda_k)},
& q=\infty.
\end{cases}
\end{equation}
For $q=\infty$, the corresponding little sequence space is
\[
\ell_0^{p,\infty}(\Lambda)
:=
\left\{
a\in\ell^{p,\infty}(\Lambda):
\|a|_{\Lambda_k}\|_{\ell^p(\Lambda_k)}
\longrightarrow0
\ \text{as }k\to\infty
\right\}.
\]
In particular,
\[
\ell^{p,p}(\Lambda)=\ell^p(\Lambda).
\]
Thus the relevant sequence space is origin-centered and annular rather than
Cartesian. The annular organization is tied to the origin and is not translation
invariant. This is one
of the main structural differences from both the classical one-exponent
Fock theory and the Cartesian mixed norms arising in time-frequency
analysis.

\smallskip
\noindent
Throughout, we use the convention $\frac{1}{\infty}=0$.

\subsection{Normalized restriction and the central question}
\label{SS:IntroductionNormalizationQuestion}

The first issue is to identify the natural scale of point evaluations.  The
structural theory developed in the companion paper \cite[Lemma~2.4]{FT26}
gives the sharp estimate
\begin{equation}
\label{eq:intro-point-evaluation}
|f(z)|e^{-\frac{\alpha}{2}|z|^2}
\lesssim_{\alpha,p,q}
(1+|z|)^{\frac{1}{p}-\frac{1}{q}}
\|f\|_{\mathcal F_\alpha^{p,q}}.
\end{equation}
Thus Gaussian normalization alone does not put point evaluations on a
uniform scale as $|z|\to\infty$. The additional polynomial factor leads
naturally to the following normalization.

\smallskip
\noindent
For a locally finite set $\Lambda\subset\mathbb C$, define the
normalized restriction map
\begin{equation}
\label{eq:intro-restriction}
R_{\alpha,\Lambda}^{p,q}f
:=
\left(
f(\lambda)e^{-\frac{\alpha}{2}|\lambda|^2}
(1+|\lambda|)^{\frac{1}{q}-\frac{1}{p}}
\right)_{\lambda\in\Lambda}.
\end{equation}
We take $\ell^{p,q}(\Lambda)$ as its natural target sequence space.

\smallskip
\noindent
We say that $\Lambda$ is \emph{sampling} for
$\mathcal F_\alpha^{p,q}$ if there exist constants
$C_1,C_2>0$, depending only on $\alpha,p,q$ and $\Lambda$, such that
\[
C_1\|f\|_{\mathcal F_\alpha^{p,q}}
\le
\|R_{\alpha,\Lambda}^{p,q}f\|_{\ell^{p,q}(\Lambda)}
\le
C_2\|f\|_{\mathcal F_\alpha^{p,q}}
\]
for all $f\in\mathcal F_\alpha^{p,q}$, and
\emph{interpolating} for $\mathcal F_\alpha^{p,q}$ if the
corresponding normalized restriction map
\[
R_{\alpha,\Lambda}^{p,q}:
\mathcal F_\alpha^{p,q}
\longrightarrow
\ell^{p,q}(\Lambda)
\]
is bounded and onto.

\smallskip
\noindent
At the outer endpoint, we say that $\Lambda$ is \emph{interpolating}
for the little space $f_\alpha^{p,\infty}$ if the corresponding
normalized restriction map
\[
R_{\alpha,\Lambda}^{p,\infty}:
f_\alpha^{p,\infty}
\longrightarrow
\ell_0^{p,\infty}(\Lambda)
\]
is bounded and onto.

\smallskip
\noindent
The normalization in \eqref{eq:intro-restriction} depends explicitly on both
mixed exponents.  The central question is therefore whether this change of
scale also changes the density law:

\medskip

\begin{center}
\emph{After point values are normalized at their natural mixed-norm scale,
does the critical sampling and interpolation density remain
$\frac{\alpha}{\pi}$, or does it depend on $p$ and $q$?}
\end{center}

\medskip
\noindent
There is no formal reason for the classical threshold to persist.  Mixed
norms change both the admissible sequence space and the global organization
of integrability.  The main results below show, however, that in the present
Gaussian setting the normalization depends on $p$ and $q$, while the
critical density remains $\frac{\alpha}{\pi}$.

\subsection{Related work and context}

Classical Fock density theory has been extended in several directions.
Sampling and interpolation for more general weighted Fock-type spaces
were studied by Berndtsson and Ortega-Cerd\`a
\cite{BerndtssonOrtega1995}, Ortega-Cerd\`a and Seip
\cite{OrtegaCerdaSeip1998}, and Marco, Massaneda, and Ortega-Cerd\`a
\cite{MarcoMassanedaOrtega2003}. For rapidly growing radial weights,
corresponding sampling and interpolation problems were investigated by
Borichev, Dhuez, and Kellay \cite{BorichevDhuezKellay2007}. Related
density results for weighted Fock spaces in several complex variables
were obtained by Lindholm \cite{Lindholm2001} and, more recently, by
Gr\"ochenig, Haimi, Ortega-Cerd\`a, and Romero
\cite{GrochenigHaimiOrtegaRomero2019}. In these settings, the relevant
density is adapted to the geometry determined by the weight. The present
problem is different: the Gaussian kernel geometry remains fixed, while
the global integrability is reorganized from a single spatial exponent
into a radial--angular pair.

\smallskip
\noindent
The closest radial--angular analogue is the sampling and interpolation
theory for mixed-norm Bergman spaces developed by Nguyen and Luecking
\cite{NguyenLuecking2018}. In that setting, the density conditions
depend on the outer exponent. Thus passing to a radial--angular mixed
norm can genuinely change the density law of the underlying
one-parameter space. The disk and Gaussian Fock settings have different
geometries and density conventions, so their numerical thresholds are
not directly comparable. The relevant point here is conceptual: the
persistence of the classical value $\frac{\alpha}{\pi}$ in the Gaussian
mixed-norm scale is a theorem, not a formal consequence of the
definition.

\smallskip
\noindent
Mixed-norm Fock spaces, together with closely related radial weighted
Fock spaces, have also been studied from
several structural and operator-theoretic viewpoints. These include
Littlewood--Paley and norm-equivalence questions, Fock projections,
Hausdorff-type operators, and random entire functions; see, for example,
\cite{BG24,CP16,FT23,Liu2024,Liu2025,
LiuShiWei2027}. These works concern the function-space structure or
specific operators, whereas the present paper studies the discrete
geometry of sampling and interpolation.

\smallskip
\noindent
We also distinguish the present spaces from another use of the term
``mixed-norm Fock space'' in time-frequency analysis. Under the
Bargmann transform, modulation spaces correspond to spaces of entire
functions with Cartesian mixed Lebesgue norms in phase-space
coordinates; see \cite{SignahlToft2012}. Those norms are adapted to
translations in the underlying time-frequency geometry. Here the mixed
norm is instead obtained by first taking angular means on Euclidean
circles and then a radial norm, leading to the origin-centered annular
sequence geometry \eqref{eq:intro-annular-sequence}.

\smallskip
\noindent
The companion paper \cite{FT26} develops the intrinsic
function-space theory used below, including point evaluations, annular
discretization, kernel estimates, approximation, and Gaussian pairing
estimates. Building on these inputs, the present paper develops
localized synthesis and stability-transfer tools for the sampling and
interpolation geometry of the mixed-norm Fock scale.

\smallskip
\noindent
To our knowledge, no previous work gives sharp sampling and
interpolation density characterizations for the radial--angular
Gaussian mixed-norm spaces $\mathcal F_\alpha^{p,q}$ throughout the
full range $0<p,q\le\infty$, including the quasi-Banach regime and the
little endpoint $f_\alpha^{p,\infty}$.

\subsection{Main results}
\label{SS:IntroductionMainResults}

We first introduce the geometric quantities entering the statements. For \(z\in\mathbb C\) and \(R>0\), let \(B(z,R)\) denote the open Euclidean disk with center \(z\) and radius \(R\). A set
$\Lambda\subset\mathbb C$ is \emph{separated} if
\[
\inf_{\substack{\lambda,\mu\in\Lambda\\ \lambda\ne\mu}}
|\lambda-\mu|>0,
\]
and \emph{relatively separated} if
\[
\sup_{z\in\mathbb C}
\#\bigl(\Lambda\cap B(z,1)\bigr)<\infty.
\]
For a locally finite set $\Lambda\subset\mathbb C$, define the standard
lower and upper Beurling densities (see, for example,
\cite{Seip1992,SeipWallsten1992,Zhu2012}) by
\[
D^-(\Lambda)
:=
\liminf_{R\to\infty}
\inf_{z\in\mathbb C}
\frac{\#(\Lambda\cap B(z,R))}{\pi R^2}
\qquad \text{and} \qquad
D^+(\Lambda)
:=
\limsup_{R\to\infty}
\sup_{z\in\mathbb C}
\frac{\#(\Lambda\cap B(z,R))}{\pi R^2}.
\]
For arbitrary locally finite sampling configurations, clusters may increase
the raw counting density without adding independent sampling information.
We therefore introduce the separated lower density
\begin{equation*}
\label{eq:intro-separated-density}
D^-_{\mathrm{sep}}(\Lambda)
:=
\sup\bigl\{
D^-(\Gamma):
\Gamma\subset\Lambda,\ \Gamma\ \text{separated}
\bigr\}.
\end{equation*}
If $\Lambda$ is separated, then
\[
D^-_{\mathrm{sep}}(\Lambda)=D^-(\Lambda).
\]

\smallskip
\noindent
Our first result gives the complete sampling characterization for arbitrary
locally finite configurations.

\begin{theorem}[Sampling theorem]
\label{thm:intro-sampling}
Let $\alpha>0$, $0<p,q\le\infty$, and let
$\Lambda\subset\mathbb C$ be locally finite. Then $\Lambda$ is sampling for
$\mathcal F_\alpha^{p,q}$ if and only if
\[
\begin{cases}
\Lambda \text{ is relatively separated and }
D^-_{\mathrm{sep}}(\Lambda)>\dfrac{\alpha}{\pi},
& 0<p<\infty,\\[2mm]
D^-_{\mathrm{sep}}(\Lambda)>\dfrac{\alpha}{\pi},
& p=\infty.
\end{cases}
\]
\end{theorem}

\smallskip
\noindent
Two features of the theorem are worth emphasizing. First,
$D^-_{\mathrm{sep}}$ measures the density of a separated sampling carrier
rather than the raw multiplicity of a clustered configuration. Second, the
local geometric condition changes at the inner endpoint $p=\infty$.
For $0<p<\infty$, nearby samples accumulate in the inner $\ell^p$-block
norm, and the upper sampling inequality forces relative separation. For
$p=\infty$, the inner block norm is a supremum, so local multiplicity is
invisible and no relative-separation condition is required.

\smallskip
\noindent
Interpolation has a uniform geometric characterization throughout the
mixed scale.

\begin{theorem}[Interpolation theorem]
\label{thm:main-interpolation}
Let $\alpha>0$, $0<p,q\le\infty$, and let
$\Lambda\subset\mathbb C$ be locally finite. Then $\Lambda$ is
interpolating for $\mathcal F_\alpha^{p,q}$ if and only if
\[
\Lambda \text{ is separated}
\qquad\text{and}\qquad
D^+(\Lambda)<\frac{\alpha}{\pi}.
\]
Moreover, for every $0<p\le\infty$, the same condition characterizes
interpolation for the little endpoint $f_\alpha^{p,\infty}$.
In both cases, the corresponding normalized restriction map admits a
bounded linear right inverse.
\end{theorem}

\smallskip
\noindent
The assertion about the right inverse is stronger than bounded
surjectivity in the quasi-Banach range. The quasi-Banach open mapping
theorem gives uniformly controlled preimages, but not a bounded linear
choice of interpolant. Under the geometric condition above, the proof constructs such a
bounded linear right inverse explicitly.

\subsection{Proof architecture}

The proof is organized around two localization principles.

\smallskip
\noindent
The first is a localized synthesis principle. Rapidly localized Hilbert
dual and Lagrange atoms remain bounded synthesis families throughout the
radial--angular mixed-norm scale, with the appropriate annular weights.
Together with the local geometry of the restriction map, this transfers
the classical Hilbert sampling and interpolation constructions to
$\mathcal F_\alpha^{p,q}$ and yields the sufficiency directions,
including the quasi-Banach range and the little endpoint.

\smallskip
\noindent
The second principle is a lower-stability transfer for rapidly localized
matrices. The main difficulty is that the annular mixed sequence norms
are tied to the origin and are not translation invariant. We overcome
this by first transferring weighted mixed lower stability to
$\ell^\infty$ through translated cutoffs and a commutator argument, and
then passing to $\ell^2$ by the $p$-independence theorem for the
Sjöstrand class.

\smallskip
\noindent
For necessity, these two streams are combined through a fixed Hilbert
sampling lattice. On the sampling side, a separated carrier retaining
the lower sampling inequality is transferred to Hilbert sampling. On
the interpolation side, bounded surjectivity, finite-support norming,
and the Gaussian pairing give a weighted lower synthesis estimate, which
is transferred to a Hilbert Riesz-sequence estimate. The classical
Hilbert Fock density theorem then yields the strict lower and upper
density conditions.

\subsection{Organization}

Section~\ref{sec:structural-inputs} collects and develops the analytic tools used throughout the
paper. It introduces the weighted mixed scales, records the required
analytic inputs from the companion paper, recalls the relevant Hilbert
Fock frame machinery, and proves the weighted localized synthesis
theorem.

\smallskip
\noindent
Section~\ref{sec:local-geometry} develops the local geometry of sampling and interpolation.
It proves upper restriction and the corresponding separation
consequences, constructs separated sampling carriers, and extends
Hilbert sampling reconstruction to the full mixed-norm scale.

\smallskip
\noindent
Section~\ref{sec:sufficiency} proves the sufficiency directions of the sampling and
interpolation theorems, including the little endpoint and the
construction of bounded linear right inverses.

\smallskip
\noindent
Section~\ref{sec:localized-stability} develops the localized lower-stability transfer for weighted
annular mixed sequence spaces. This discrete matrix argument is
independent of the Fock-space structure, apart from the annular
sequence-space framework.

\smallskip
\noindent
Section~\ref{sec:mixed-hilbert-transfer} combines the analytic and geometric tools of Sections~\ref{sec:structural-inputs} and \ref{sec:local-geometry}
with the stability-transfer principle of Section~\ref{sec:localized-stability} to obtain the
mixed-to-Hilbert sampling and interpolation transfers.

\smallskip
\noindent
Finally, Section~\ref{sec:necessity} proves the strict sampling and interpolation
necessity statements and thereby completes the proofs of the main
theorems.


\section{Analytic tools and structural inputs}
\label{sec:structural-inputs}

The basic mixed-norm spaces, annular sequence spaces, normalized
restriction map, and density notions were introduced in Section~\ref{S:Introduction}.
In this section we collect and develop the analytic tools used throughout the
paper. After introducing the weighted scales and recalling the
required inputs from the companion paper, we record the Hilbert
frame machinery and establish the weighted localized synthesis
theorem.

\smallskip
\noindent
Throughout, $\mathbb N_0=\{0,1,2,\ldots\}$, $dA$ denotes planar
Lebesgue measure, and $H(\mathbb C)$ denotes the space of entire
functions. For an index set $Z$, let $c_{00}(Z)$ denote the space of
finitely supported complex sequences indexed by $Z$.

\smallskip
\noindent
We write $A\lesssim B$ if $A\le CB$, where the implicit constant
$C>0$ is independent of the variables under consideration.
Dependence on parameters is indicated by subscripts, while dependence
on fixed sets, families, or operators is usually suppressed.
Once specified, the parameter dependence may be omitted in subsequent
estimates within the same argument. We write $A\asymp B$ if
$A\lesssim B$ and $B\lesssim A$.

\subsection{Additional notation and weighted scales}
\label{SS:AdditionalNotation}
In this subsection, $\sigma,\nu\in\mathbb R$ denote generic weight
parameters. For $0<p,q\le\infty$, the distinguished mixed-norm weight
exponent will be denoted by
\[
\nu(p,q):=\frac1q-\frac1p.
\]
For $\sigma\in\mathbb R$, define the polynomially weighted mixed-norm
Fock space $\mathcal F_{\alpha;\sigma}^{p,q}$ by
\[
\|f\|_{\mathcal F_{\alpha;\sigma}^{p,q}}
:=
\begin{cases}
\displaystyle
\left(
\alpha q\int_0^\infty
\bigl[M_p(f,r)(1+r)^\sigma\bigr]^q
e^{-\frac{\alpha q}{2}r^2}r\,dr
\right)^{\frac{1}{q}},
& 0<q<\infty,\\[4mm]
\displaystyle
\sup_{r\ge0}
M_p(f,r)e^{-\frac{\alpha}{2}r^2}(1+r)^\sigma,
& q=\infty.
\end{cases}
\]
Thus
\[
\mathcal F_{\alpha;0}^{p,q}=\mathcal F_\alpha^{p,q}.
\]
For $\nu\in\mathbb R$, set
\[
w_\nu(k):=(1+k)^\nu,\qquad k\in\mathbb N_0.
\]
For a locally finite set $\Lambda\subset\mathbb C$, let
$\ell_{w_\nu}^{p,q}(\Lambda)$ be the space of all sequences
$a=(a_\lambda)_{\lambda\in\Lambda}$ such that
\[
\|a\|_{\ell_{w_\nu}^{p,q}(\Lambda)}
:=
\left\|
\left(
w_\nu(k)\,
\|a|_{\Lambda_k}\|_{\ell^p(\Lambda_k)}
\right)_{k\ge0}
\right\|_{\ell^q}
<\infty.
\]
When $q=\infty$, let $\ell_{w_\nu,0}^{p,\infty}(\Lambda)$ be the
subspace of those $a\in\ell_{w_\nu}^{p,\infty}(\Lambda)$ for which
\[
w_\nu(k)\,
\|a|_{\Lambda_k}\|_{\ell^p(\Lambda_k)}
\longrightarrow0
\qquad (k\to\infty).
\]
For \(\nu = 0\),
\[
\ell_{w_0}^{p,q}(\Lambda)=\ell^{p,q}(\Lambda),
\qquad
\ell_{w_0,0}^{p,\infty}(\Lambda)=\ell_0^{p,\infty}(\Lambda).
\]
For later use, let $0<p,q\le\infty$ and $\sigma\in\mathbb R$, and define
\[
\ell_{\mathrm{nat};\sigma}^{p,q}(\Lambda)
:=
\ell_{w_{\nu(p,q)+\sigma}}^{p,q}(\Lambda).
\]
In particular,
\[
\ell_{\mathrm{nat}}^{p,q}(\Lambda)
:=
\ell_{\mathrm{nat};0}^{p,q}(\Lambda)
=
\ell_{w_{\nu(p,q)}}^{p,q}(\Lambda).
\]
At the outer endpoint, we also write
\[
\ell_{\mathrm{nat};\sigma,0}^{p,\infty}(\Lambda)
:=
\ell_{w_{\nu(p,\infty)+\sigma},0}^{p,\infty}(\Lambda),
\]
and, in particular,
\[
\ell_{\mathrm{nat},0}^{p,\infty}(\Lambda)
:=
\ell_{\mathrm{nat};0,0}^{p,\infty}(\Lambda)
=
\ell_{w_{\nu(p,\infty)},0}^{p,\infty}(\Lambda).
\]
As in \cite[Remark~5.28]{FT26}, the spaces
$\ell_{w_\nu}^{p,q}(\Lambda)$ are complete quasi-Banach spaces and
their quasi-norms satisfy the powered triangle inequality with exponent
$s=\min\{1,p,q\}$. The little spaces
$\ell_{w_\nu,0}^{p,\infty}(\Lambda)$ and $f_\alpha^{p,\infty}$ are
closed in their respective ambient spaces.

\smallskip
\noindent
Since
\[
1+k\le 1+|\lambda|<2(1+k),
\qquad \lambda\in\Lambda_k,
\]
pointwise and shell polynomial weights are interchangeable, up to
constants depending only on the exponent. We shall use this comparison
without further comment.

\smallskip
\noindent
For $0<p\le\infty$, define the generalized conjugate exponent $p^*$ by
\[
p^*
:=
\begin{cases}
\infty, & 0<p\le1,\\[1mm]
\dfrac{p}{p-1}, & 1<p<\infty,\\[2mm]
1, & p=\infty.
\end{cases}
\]
Equivalently,
\( \displaystyle 
\frac1{p^*}
=
\left(1-\frac1p\right)_+.
\)
We define $q^*$ analogously and set
\[
\sigma(p,q)
:=
\left(\frac1p-1\right)_+
-
\left(\frac1q-1\right)_+.
\]
The elementary identity
\begin{equation}
\label{eq:duality-correction}
\nu(p,q)
= - \nu(p^*, q^*)-\sigma(p,q)
\end{equation}
will be used in the interpolation argument.

\smallskip
\noindent
For $w\in\mathbb C$, let
\[
\kappa_{\alpha,w}(z)
:=
e^{\alpha z\overline w-\frac{\alpha}{2}|w|^2}.
\]
Then
\begin{equation*}
\label{eq:normalized-kernel-decay}
|\kappa_{\alpha,w}(z)|
e^{-\frac{\alpha}{2}|z|^2}
=
e^{-\frac{\alpha}{2}|z-w|^2}.
\end{equation*}
Whenever the integral is absolutely convergent, we use the Gaussian
pairing
\begin{equation}
\label{eq:gaussian-pairing}
\langle f,g\rangle_\alpha
:=
\frac{\alpha}{\pi}
\int_{\mathbb C}
f(z)\overline{g(z)}
e^{-\alpha|z|^2}\,dA(z),
\end{equation}
which is linear in the first variable and conjugate-linear in the second.

\subsection{Analytic inputs from the companion paper}
\label{SS:StructuralInputsFT26}

We record the analytic results from the companion paper
\cite{FT26} that will be used below.

\smallskip
\noindent
(F1) \emph{Quasi-Banach structure.}
For $0<p,q\le\infty$, set
\[
s:=\min\{1,p,q\}.
\]
Then $\mathcal F_\alpha^{p,q}$ is complete and
\[
\|f+g\|_{\mathcal F_\alpha^{p,q}}^s
\le
\|f\|_{\mathcal F_\alpha^{p,q}}^s
+
\|g\|_{\mathcal F_\alpha^{p,q}}^s.
\]
See \cite[Lemma~2.13]{FT26}.

\smallskip
\noindent
(F2) \emph{Weighted annular discretization.}
For $\sigma\in\mathbb R$, set
\[
B_{p,\sigma}(f;k)
:=
\sup_{r\in[k,k+1)}
M_p(f,r)e^{-\frac{\alpha}{2}r^2}(1+r)^\sigma .
\]
For convenience, we write \(B_p(f;k):=B_{p,0}(f;k)\).
Then, for $0<q<\infty$,
\[
\|f\|_{\mathcal F_{\alpha;\sigma}^{p,q}}^q
\asymp_{\alpha,\sigma,p,q}
\sum_{k=0}^\infty
(1+k)B_{p,\sigma}(f;k)^q,
\]
whereas
\[
\|f\|_{\mathcal F_{\alpha;\sigma}^{p,\infty}}
=
\sup_{k\ge0}B_{p,\sigma}(f;k).
\]
See \cite[Proposition~2.6 and Lemma~7.1]{FT26}.

\smallskip
\noindent
(F3) \emph{Normalized-kernel size.}
For every $w\in\mathbb C$,
\[
\|\kappa_{\alpha,w}\|_{\mathcal F_\alpha^{p,q}}
\asymp_{\alpha,p,q}
(1+|w|)^{\nu(p,q)}.
\]
See \cite[Proposition~2.11]{FT26}.

\smallskip
\noindent
(F4) \emph{Dilation approximation.}
For $0<\tau<1$, let
\[
D_\tau f(z)=f(\tau z).
\]
If $0<q<\infty$, then
\[
\|D_\tau f-f\|_{\mathcal F_\alpha^{p,q}}
\longrightarrow0
\quad \text{as} \quad \tau\uparrow1.
\]
See \cite[Lemma~2.14]{FT26}.

\smallskip
\noindent
(F5) \emph{Forward Gaussian pairing.}
Let $0<p,q\le\infty$. Then the Gaussian pairing is absolutely
convergent and
\[
|\langle f,g\rangle_\alpha|
\lesssim_{\alpha,p,q}
\|f\|_{\mathcal F_\alpha^{p,q}}
\,
\|g\|_{\mathcal F_{\alpha;\sigma(p,q)}^{p^*,q^*}}
\]
for every
\(
f\in\mathcal F_\alpha^{p,q}\) and \(
g\in\mathcal F_{\alpha;\sigma(p,q)}^{p^*,q^*}\).
The case $0<q<\infty$ was proved in
\cite[Proposition~7.4]{FT26}. The case $q=\infty$ follows from the
argument in \cite[Proposition~7.15, Step~1]{FT26}, which uses only the
$\mathcal F_\alpha^{p,\infty}$ norm and hence applies to the full space
$\mathcal F_\alpha^{p,\infty}$.

\subsection{Hilbert Fock frames and localized atoms}
\label{SS:HilbertFock}

For $p=q=2$, the Gaussian pairing \eqref{eq:gaussian-pairing} is the
inner product of $\mathcal F_\alpha^2$.  If $Z\subset\mathbb C$ is
separated, define
\[
C_Z:\mathcal F_\alpha^2\longrightarrow \ell^2(Z),
\qquad
(C_Zf)_z
:=
\langle f,\kappa_{\alpha,z}\rangle_\alpha
=
f(z)e^{-\frac{\alpha}{2}|z|^2},
\]
and
\[
C_Z^*c
=
\sum_{z\in Z}c_z\kappa_{\alpha,z}.
\]
We write
\[
S_Z:=C_Z^*C_Z,
\qquad
G_Z:=C_ZC_Z^*
\]
for the corresponding frame operator and Gramian.

\smallskip
\noindent
The Gramian satisfies
\begin{equation}
\label{eq:hilbert-gramian-decay}
\left|
(G_Z)_{z,w}
\right|
=
\left|
\langle\kappa_{\alpha,w},\kappa_{\alpha,z}\rangle_\alpha
\right|
=
e^{-\frac{\alpha}{2}|z-w|^2}.
\end{equation}
Since $Z$ is separated, the normalized kernels
$(\kappa_{\alpha,z})_{z\in Z}$ form a Bessel sequence in
$\mathcal F_\alpha^2$; see, for example,
\cite[Lemma~4.9]{Zhu2012}. Hence
\[
C_Z:\mathcal F_\alpha^2\to\ell^2(Z),
\qquad
C_Z^*:\ell^2(Z)\to\mathcal F_\alpha^2
\]
are bounded.

\smallskip
\noindent
In the Hilbert setting, the sampling inequality is exactly the frame
inequality for the normalized kernels. Likewise, if $Z$ is separated, then $C_Z$ is bounded, and
interpolation is equivalent to surjectivity of $C_Z$. Since
$C_Z$ is onto if and only if $C_Z^*$ is bounded below, this is
equivalent to $(\kappa_{\alpha,z})_{z\in Z}$ being a Riesz
sequence.

\begin{theorem}[Classical Hilbert Fock density theorem]
\label{thm:classical-hilbert-density}
Let $\alpha>0$.
\begin{enumerate}
\item[\textnormal{(i)}]
A separated set $\Gamma\subset\mathbb C$ is sampling for
$\mathcal F_\alpha^2$ if and only if
\[
D^-(\Gamma)>\frac{\alpha}{\pi}.
\]

\item[\textnormal{(ii)}]
A locally finite set $\Lambda\subset\mathbb C$ is interpolating for
$\mathcal F_\alpha^2$ if and only if it is separated and
\[
D^+(\Lambda)<\frac{\alpha}{\pi}.
\]
\end{enumerate}
\end{theorem}

\noindent
These are the classical density theorems of Seip and Seip--Wallst\'en;
see \cite[Corollaries~4.37 and~4.51]{Zhu2012} for the above formulation,
and \cite{Seip1992,SeipWallsten1992} for the original results.

\medskip
\noindent
We next record the localization consequence needed below.
For a separated set $Z\subset\mathbb C$ and $L>0$, let
$\mathcal A_L(Z)$ denote the polynomially weighted Schur class of
matrices $A=(A_{z,w})_{z,w\in Z}$ satisfying
\[
\sup_{z\in Z}
\sum_{w\in Z}
|A_{z,w}|(1+|z-w|)^L
+
\sup_{w\in Z}
\sum_{z\in Z}
|A_{z,w}|(1+|z-w|)^L
<\infty.
\]
By \eqref{eq:hilbert-gramian-decay},
\[
G_Z\in\mathcal A_L(Z)
\qquad (L>0).
\]
Since $Z$ is separated, $(Z,|\cdot|)$ with counting measure has
polynomial growth. For every $L>0$, the polynomial weight
$(1+|z-w|)^L$ satisfies the admissibility assumptions of
\cite[Theorem~4.1]{Sun2007}; see \cite[Example~A.2]{Sun2007}.
Hence, as $G_Z\in\mathcal A_L(Z)$, inverse-closedness gives
\begin{equation}\label{eq:localized-gramian-inverse}
G_Z^{-1}\in\mathcal A_L(Z)
\end{equation}
whenever $G_Z$ is invertible on $\ell^2(Z)$.

\smallskip
\noindent
If $\Gamma$ is sampling for $\mathcal F_\alpha^2$, let
$A_\Gamma,B_\Gamma$ be frame bounds. Then
\[
A_\Gamma I\le S_\Gamma\le B_\Gamma I,
\]
and the nonzero spectra of $S_\Gamma=C_\Gamma^*C_\Gamma$ and
$G_\Gamma=C_\Gamma C_\Gamma^*$ coincide. Hence
\[
\operatorname{spec}(G_\Gamma)
\subset\{0\}\cup[A_\Gamma,B_\Gamma],
\]
so $0$ is isolated in the spectrum. The Moore--Penrose inverse is
therefore given by
\[
G_\Gamma^\dagger
=
\frac{1}{2\pi i}
\int_{\mathcal C}
\zeta^{-1}(\zeta I-G_\Gamma)^{-1}\,d\zeta,
\]
where $\mathcal C$ surrounds the nonzero spectrum and excludes $0$.
For $\zeta\in\mathcal C$, inverse-closedness gives
\[
(\zeta I-G_\Gamma)^{-1}\in\mathcal A_L(\Gamma),
\]
and hence
\begin{equation}\label{eq:localized-gramian-pseudoinverse}
G_\Gamma^\dagger\in\mathcal A_L(\Gamma),\qquad L>0.
\end{equation}
This is the same localized range-inverse argument as in
\cite[Proposition~4.10]{FT26}.

\begin{proposition}
\label{prop:localized-hilbert-atoms}
The following assertions hold.

\begin{enumerate}
\item[\textnormal{(i)}]
If $\Gamma\subset\mathbb C$ is separated and sampling for
$\mathcal F_\alpha^2$, then its canonical dual atoms
\[
\varphi_\gamma
:=
S_\Gamma^{-1}\kappa_{\alpha,\gamma},
\qquad \gamma\in\Gamma,
\]
satisfy, for every $N>0$,
\begin{equation}
\label{eq:dual-atom-localization}
|\varphi_\gamma(z)|
e^{-\frac{\alpha}{2}|z|^2}
\lesssim_{\alpha,N}
(1+|z-\gamma|)^{-N},
\qquad z\in\mathbb C.
\end{equation}
Moreover,
\begin{equation}
\label{eq:hilbert-frame-reconstruction}
f
=
\sum_{\gamma\in\Gamma}
\langle f,\kappa_{\alpha,\gamma}\rangle_\alpha
\varphi_\gamma,
\qquad
f\in\mathcal F_\alpha^2,
\end{equation}
with convergence in $\mathcal F_\alpha^2$.

\item[\textnormal{(ii)}]
If $\Lambda\subset\mathbb C$ is separated and interpolating for
$\mathcal F_\alpha^2$, then there exist
$\psi_\lambda\in\mathcal F_\alpha^2$, $\lambda\in\Lambda$, such that
\begin{equation}
\label{eq:lagrange-property}
\psi_\lambda(\mu)
e^{-\frac{\alpha}{2}|\mu|^2}
=
\delta_{\lambda,\mu},
\qquad
\lambda,\mu\in\Lambda,
\end{equation}
and, for every $N>0$,
\begin{equation}
\label{eq:lagrange-atom-localization}
|\psi_\lambda(z)|
e^{-\frac{\alpha}{2}|z|^2}
\lesssim_{\alpha,N}
(1+|z-\lambda|)^{-N},
\qquad z\in\mathbb C.
\end{equation}
\end{enumerate}
\end{proposition}

\begin{proof}
For the sampling case, the canonical dual synthesis satisfies
\[
S_\Gamma^{-1}C_\Gamma^*
=
C_\Gamma^*G_\Gamma^\dagger.
\]
Indeed, both sides vanish on
\(
(\operatorname{Ran}C_\Gamma)^\perp=\ker C_\Gamma^*,
\)
while on $\operatorname{Ran}C_\Gamma$ the identity follows from
\(
G_\Gamma C_\Gamma=C_\Gamma S_\Gamma.
\)
Consequently,
\begin{equation}
\label{eq:dual-atom-expansion}
\varphi_\gamma
=
\sum_{\mu\in\Gamma}
(G_\Gamma^\dagger)_{\mu,\gamma}
\kappa_{\alpha,\mu}.
\end{equation}
The reconstruction formula
\eqref{eq:hilbert-frame-reconstruction} is the canonical frame
expansion.

\smallskip
\noindent
If $\Lambda$ is interpolating, the normalized kernels form a Riesz
sequence, so $G_\Lambda$ is invertible on $\ell^2(\Lambda)$. Define
\begin{equation}
\label{eq:lagrange-atom-expansion}
\psi_\lambda
:=
\sum_{\mu\in\Lambda}
(G_\Lambda^{-1})_{\mu,\lambda}
\kappa_{\alpha,\mu}.
\end{equation}
Then
\[
C_\Lambda\psi_\lambda
=
G_\Lambda G_\Lambda^{-1}e_\lambda
=
e_\lambda,
\]
which gives \eqref{eq:lagrange-property}. The series in
\eqref{eq:dual-atom-expansion} and
\eqref{eq:lagrange-atom-expansion} converge in $\mathcal F_\alpha^2$,
since the corresponding coefficient columns belong to $\ell^2$ and
the kernel synthesis operators are bounded.

\smallskip
\noindent
It remains to prove localization. In either case, let $Z$ denote the
corresponding index set, and write the expansion in the form
\[
h_\eta
=
\sum_{\mu\in Z}
a_{\mu,\eta}\kappa_{\alpha,\mu},
\qquad \eta\in Z,
\]
where $(Z,h_\eta)$ stands for either
$(\Gamma,\varphi_\gamma)$ or $(\Lambda,\psi_\lambda)$.
By \eqref{eq:localized-gramian-inverse} and
\eqref{eq:localized-gramian-pseudoinverse}, for every $L>0$,
\[
|a_{\mu,\eta}|
\lesssim_{\alpha,L}
(1+|\mu-\eta|)^{-L}.
\]
Hence
\[
|h_\eta(z)|e^{-\frac{\alpha}{2}|z|^2}
\lesssim
\sum_{\mu\in Z}
(1+|\mu-\eta|)^{-L}
e^{-\frac{\alpha}{2}|z-\mu|^2}.
\]
Fix $N>0$ and choose $L\ge N$. Since
\[
1+|z-\eta|
\le
(1+|z-\mu|)(1+|\mu-\eta|),
\]
we obtain
\[
\begin{aligned}
(1+|z-\eta|)^N
|h_\eta(z)|e^{-\frac{\alpha}{2}|z|^2}
& \, \lesssim
\sum_{\mu\in Z}
(1+|\mu-\eta|)^{N-L}
(1+|z-\mu|)^N
e^{-\frac{\alpha}{2}|z-\mu|^2}
\\
&\, \lesssim
\sum_{\mu\in Z}
(1+|z-\mu|)^N
e^{-\frac{\alpha}{2}|z-\mu|^2}.
\end{aligned}
\]
The last sum is uniformly bounded in $z$ by separated-set counting.
This proves \eqref{eq:dual-atom-localization} and
\eqref{eq:lagrange-atom-localization}. The same majorant also gives
absolute locally uniform convergence of the expansions.
\end{proof}

\medskip
\noindent
For a Hilbert sampling set $\Gamma$, we write
\[
S_\Gamma^\varphi c
:=
\sum_{\gamma\in\Gamma}c_\gamma\varphi_\gamma
\]
for synthesis by the canonical dual atoms.  Thus
\[
S_\Gamma^\varphi C_\Gamma=I
\qquad\text{on }\mathcal F_\alpha^2.
\]

\subsection{Weighted localized synthesis}
\label{subsec:localized-synthesis}

We now establish the localized synthesis tool used in both the
sufficiency and necessity arguments. We begin with the one-shell
estimate underlying the weighted synthesis theorem.

\begin{definition}[Rapid Fock localization]
\label{def:rapid-fock-localization}
Let $\Sigma\subset\mathbb C$ be separated. A family
$(\phi_\xi)_{\xi\in\Sigma}\subset H(\mathbb C)$ is called \textit{rapidly Fock
localized} if, for every $N>0$, there is a constant $C_N>0$ such that
\begin{equation}
\label{eq:rapid-fock-localization}
|\phi_\xi(z)|e^{-\frac{\alpha}{2}|z|^2}
\le
C_N(1+|z-\xi|)^{-N},
\qquad
z\in\mathbb C,\ \xi\in\Sigma.
\end{equation}
\end{definition}

\smallskip
\noindent
For $m\in\mathbb N_0$, write
\[
\Sigma_m:=\Sigma\cap A_m.
\]
Let $d_\Sigma>0$ be a separation constant for $\Sigma$. Since the disks $B(\xi,d_\Sigma/2)$, $\xi\in\Sigma$, are pairwise
disjoint, a standard packing argument in unit-width annuli gives,
with constants depending only on the separation data of $\Sigma$,
\begin{equation}
\label{eq:separated-shell-counting}
\#\Sigma_m\lesssim1+m,
\qquad
m\in\mathbb N_0.
\end{equation}
The same argument, applied to annuli centered at an arbitrary
$z\in\mathbb C$, yields
\[
\#\{\xi\in\Sigma:j\le |z-\xi|<j+1\}
\lesssim1+j
\]
uniformly in $z$. Consequently, for every $N>2$,
\begin{equation}
\label{eq:separated-decay-sum}
\sup_{z\in\mathbb C}
\sum_{\xi\in\Sigma}(1+|z-\xi|)^{-N}
<\infty.
\end{equation}

\begin{lemma}
\label{lem:one-shell-synthesis}
Let $0<p\le\infty$, let $\Sigma\subset\mathbb C$ be separated, and
let $(\phi_\xi)_{\xi\in\Sigma}$ be rapidly Fock localized. For
$m\in\mathbb N_0$, set
\[
F_m(z)
:=
\sum_{\xi\in\Sigma_m}c_\xi\phi_\xi(z),
\qquad
c_m
:=
\|c|_{\Sigma_m}\|_{\ell^p(\Sigma_m)}.
\]
Then, for every $L>0$,
\begin{equation}
\label{eq:one-shell-synthesis}
M_p(F_m,r)e^{-\frac{\alpha}{2}r^2}
\lesssim_{\alpha,p,L}
c_m(1+m)^{-\frac1p}(1+|r-m|)^{-L},
\qquad
r\ge0.
\end{equation}
\end{lemma}

\begin{proof}
We first record the angular estimates underlying the argument. If
$0<p<\infty$ and $N>L+\frac{1}{p}$, then
\begin{equation}
\label{eq:angular-shell-decay}
\left(
\frac1{2\pi}
\int_0^{2\pi}
(1+|re^{i\theta}-\xi|)^{-Np}\,d\theta
\right)^{\frac1p}
\lesssim_{p, N}
(1+m)^{-\frac1p}(1+|r-m|)^{-L}
\end{equation}
for $\xi\in\Sigma_m$. Indeed, if $\xi=0$ the estimate is immediate.
Otherwise, write
\[
\xi=\rho e^{i\theta_0},
\qquad
m\le\rho<m+1.
\]
We first claim that, for $M>1$, the identity
\[
|re^{i\theta}-\rho e^{i\theta_0}|^2
=
(r-\rho)^2
+
4r\rho\sin^2\frac{\theta-\theta_0}{2}
\]
gives
\begin{equation}
\label{eq:angular-integral-estimate}
\frac1{2\pi}
\int_0^{2\pi}
\Bigl(1+|re^{i\theta}-\rho e^{i\theta_0}|\Bigr)^{-M}\,d\theta
\lesssim_M
(1+\rho)^{-1}(1+|r-\rho|)^{-M+1}.
\end{equation}
Indeed, by a translation of the angular variable and periodicity
we may replace $\theta-\theta_0$ by $t\in[-\pi,\pi]$. If
$r\asymp\rho$ and $\rho\ge1$, then, by symmetry and
$\sin(\frac{t}{2})\gtrsim t$ for $0\le t\le\pi$,
\[
|re^{it}-\rho|
\gtrsim
\Bigl(|r-\rho|^2+\rho^2t^2\Bigr)^{\frac12}.
\]
After the change of variables $u=\rho t$, the angular integral is
bounded by
\[
\frac1{1+\rho}
\int_0^\infty
\left(1+\sqrt{|r-\rho|^2+u^2}\right)^{-M}\,du
\lesssim_M
(1+\rho)^{-1}(1+|r-\rho|)^{-M+1},
\]
where the last estimate follows by splitting the integral at
$u=1+|r-\rho|$. If $r\asymp\rho$ and $\rho<1$, the same bound is
immediate, since \(r,\rho\), and \(|r-\rho|\) are then uniformly bounded. If $r$ and $\rho$ are not comparable, then
$|r-\rho|\asymp r+\rho$, and
\(
|re^{i\theta}-\rho e^{i\theta_0}|\ge |r-\rho|
\)
gives the estimate directly. 
Applying
\eqref{eq:angular-integral-estimate} with $M=Np$ and using
\(1+\rho\asymp1+m\) and \(1+|r-\rho|\asymp1+|r-m|
\)
proves \eqref{eq:angular-shell-decay}.

\smallskip
\noindent
We shall also use the shell-sum estimate
\begin{equation}
\label{eq:angular-shell-sum}
\sup_{0\le\theta<2\pi}
\sum_{\xi\in\Sigma_m}
(1+|re^{i\theta}-\xi|)^{-N}
\lesssim
(1+|r-m|)^{-L},
\end{equation}
valid whenever $N>L+2$. Indeed, for $z=re^{i\theta}$ and $\xi\in\Sigma_m$, we have
\[
|z-\xi|
\ge
\operatorname{dist}(r,[m,m+1]).
\]
Hence
\[
\begin{aligned}
\sum_{\xi\in\Sigma_m}(1+|z-\xi|)^{-N}
&\le
\bigl(1+\operatorname{dist}(r,[m,m+1])\bigr)^{-L}
\sum_{\xi\in\Sigma}(1+|z-\xi|)^{-N+L}\\
&\lesssim
(1+|r-m|)^{-L}.
\end{aligned}
\]
Here \eqref{eq:separated-decay-sum} applies because $N-L>2$, and we
used
\[
1+\operatorname{dist}(r,[m,m+1])
\asymp
1+|r-m|.
\]
This proves \eqref{eq:angular-shell-sum}.

\smallskip
\noindent
Suppose first that $0<p<1$. Choose the localization order $N$
sufficiently large. By $p$-subadditivity,
\eqref{eq:rapid-fock-localization}, and
\eqref{eq:angular-shell-decay},
\begin{align*}
M_p(F_m,r)^p e^{-\frac{p\alpha}{2}r^2}
&\le
\sum_{\xi\in\Sigma_m}
|c_\xi|^p
\frac1{2\pi}
\int_0^{2\pi}
|\phi_\xi(re^{i\theta})|^p
e^{-\frac{p\alpha}{2}r^2}\,d\theta\\
&\lesssim
c_m^p(1+m)^{-1}(1+|r-m|)^{-Lp}.
\end{align*}
Taking the $p$-th root proves \eqref{eq:one-shell-synthesis}.

\smallskip
\noindent
Suppose now that $1\le p\le\infty$. For fixed $m$ and $r$, define
\[
U_{m,r}a(\theta)
:=
\sum_{\xi\in\Sigma_m}
a_\xi\phi_\xi(re^{i\theta})
e^{-\frac{\alpha}{2}r^2},
\]
where the angular measure is normalized. Choosing the localization
order sufficiently large, \eqref{eq:angular-shell-decay} with $p=1$
gives
\[
\|U_{m,r}\|_{\ell^1(\Sigma_m)\to L^1(\mathbb T)}
\lesssim
(1+m)^{-1}(1+|r-m|)^{-L},
\]
while \eqref{eq:angular-shell-sum} gives
\[
\|U_{m,r}\|_{\ell^\infty(\Sigma_m)\to L^\infty(\mathbb T)}
\lesssim
(1+|r-m|)^{-L}.
\]
Interpolation yields
\[
\|U_{m,r}\|_{\ell^p(\Sigma_m)\to L^p(\mathbb T)}
\lesssim
(1+m)^{-\frac1p}(1+|r-m|)^{-L},
\]
and the proof is complete.
\end{proof}

\begin{theorem}[Weighted localized synthesis]
\label{thm:weighted-localized-synthesis}
Let $0<p,q\le\infty$, let $\sigma\in\mathbb R$, and let
$\Sigma\subset\mathbb C$ be separated. Suppose that
$(\phi_\xi)_{\xi\in\Sigma}$ is rapidly Fock localized.
Then the synthesis operator
\[
S_\Sigma^\phi:
\ell_{\mathrm{nat};\sigma}^{p,q}(\Sigma)
\longrightarrow
\mathcal F_{\alpha;\sigma}^{p,q},
\qquad
S_\Sigma^\phi c
:=
\sum_{\xi\in\Sigma}c_\xi\phi_\xi,
\]
is well defined and bounded. More precisely,
\begin{equation*}
\label{eq:weighted-localized-synthesis}
\|S_\Sigma^\phi c\|_{\mathcal F_{\alpha;\sigma}^{p,q}}
\lesssim_{\alpha,p,q,\sigma}
\|c\|_{\ell_{\mathrm{nat};\sigma}^{p,q}(\Sigma)},
\qquad
c\in \ell_{\mathrm{nat};\sigma}^{p,q}(\Sigma),
\end{equation*}
and the defining series converges absolutely and locally uniformly. The implicit constant may also depend on the separation data of
$\Sigma$ and the localization constants of
$(\phi_\xi)_{\xi\in\Sigma}$.

\smallskip
\noindent
If $q=\infty$ and
$c\in \ell_{\mathrm{nat};\sigma,0}^{p,\infty}(\Sigma)$, then
\begin{equation}
\label{eq:weighted-localized-synthesis-vanishing}
(1+r)^\sigma
M_p(S_\Sigma^\phi c,r)e^{-\frac{\alpha}{2}r^2}
\longrightarrow0
\qquad (r\to\infty).
\end{equation}
\end{theorem}

\begin{proof}
For simplicity, throughout this proof, we set
\[
\nu:=\nu(p,q)+\sigma
=\frac1q-\frac1p+\sigma.
\]
Then 
\[\ell_{\mathrm{nat};\sigma}^{p,q}(\Sigma) = \ell_{w_\nu}^{p,q}(\Sigma), \qquad \ell_{\mathrm{nat};\sigma,0}^{p,\infty}(\Sigma) = \ell_{w_\nu,0}^{p,\infty}(\Sigma).
\]
For $m\in\mathbb N_0$, put
\[
c_m
:=
\|c|_{\Sigma_m}\|_{\ell^p(\Sigma_m)}.
\]

\smallskip
\noindent
We first show that $S_\Sigma^\phi c$ is well defined. Let
$K\subset\mathbb C$ be compact. By rapid localization and
\eqref{eq:separated-shell-counting}, for every $N>0$,
\[
\sup_{z\in K}|\phi_\xi(z)|
\lesssim_{\alpha,K,N}
(1+|\xi|)^{-N}
\qquad \text{and} \qquad
\sum_{\xi\in\Sigma_m}|c_\xi|
\lesssim
(\#\Sigma_m)^{\frac{1}{p^*}}c_m
\lesssim
(1+m)^{\frac{1}{p^*}}c_m.
\]
Moreover,
\[
c_m
\le
(1+m)^{-\nu}
\|c\|_{\ell_{w_\nu}^{p,q}(\Sigma)}.
\]
Consequently,
\[
\sup_{z\in K}
\sum_{\xi\in\Sigma_m}|c_\xi\phi_\xi(z)|
\lesssim
\|c\|_{\ell_{w_\nu}^{p,q}(\Sigma)}
(1+m)^{-N+\frac{1}{p^*}-\nu}.
\]
Choosing $N$ sufficiently large gives a summable majorant. Hence the
defining series for $S_\Sigma^\phi c$ converges absolutely and locally
uniformly and defines an entire function
\[
F:=S_\Sigma^\phi c.
\]

\smallskip
\noindent
It remains to prove the norm estimate. Let
\[
s:=\min\{1,p,q\}
\qquad\text{and}\qquad
F_m
:=
\sum_{\xi\in\Sigma_m}c_\xi\phi_\xi,
\quad m\ge0.
\]
For a finite set $E\subset\mathbb N_0$, put
\[
F_E:=\sum_{m\in E}F_m.
\]
Minkowski's inequality when $p\ge1$ and $p$-subadditivity when
$0<p<1$, together with $s\le1$ and $s\le p$, give
\[
\left(
M_p(F_E,r)e^{-\frac{\alpha}{2}r^2}
\right)^s
\le
\sum_{m\in E}
\left(
M_p(F_m,r)e^{-\frac{\alpha}{2}r^2}
\right)^s.
\]
Hence Lemma~\ref{lem:one-shell-synthesis} and
$1+|r-m|\asymp1+|n-m|$ for $r\in[n,n+1)$ yield
\[
B_p(F_E;n)^s
\lesssim
\sum_{m\in E}
(1+|n-m|)^{-Ls}
(1+m)^{-\frac{s}{p}}c_m^s.
\]
Applying this to the finite initial shell sums and then passing to the
locally uniform limit gives
\begin{equation}
\label{eq:weighted-synthesis-block}
B_p(F;n)^s
\lesssim
\sum_{m=0}^\infty
(1+|n-m|)^{-Ls}
(1+m)^{-\frac{s}{p}}c_m^s.
\end{equation}

\smallskip
\noindent
Suppose first that $0<q<\infty$. Take
\[
L:=\left|\sigma+\frac1q\right|+\frac2s.
\]
Since
\(
\nu=\frac1q-\frac1p+\sigma,
\)
multiplying \eqref{eq:weighted-synthesis-block} by
$(1+n)^{s(\sigma+\frac1q)}$ gives
\[
\Bigl((1+n)^{\sigma+\frac1q}B_p(F;n)\Bigr)^s
\lesssim
\sum_{m=0}^\infty
(1+|n-m|)^{-Ls}
\left(\frac{1+n}{1+m}\right)^{s(\sigma+\frac1q)}
\bigl((1+m)^\nu c_m\bigr)^s.
\]
Using
\[
\left(\frac{1+n}{1+m}\right)^{\sigma+\frac1q}
\le
(1+|n-m|)^{|\sigma+\frac1q|},
\]
we obtain
\begin{equation*}
\label{eq:weighted-synthesis-convolution}
\Bigl((1+n)^{\sigma+\frac1q}B_p(F;n)\Bigr)^s
\lesssim
\sum_{m=0}^\infty
(1+|n-m|)^{-2}
\bigl((1+m)^\nu c_m\bigr)^s.
\end{equation*}
Since $\frac{q}{s}\ge1$, Young's convolution inequality on
$\ell^{\frac{q}{s}}(\mathbb Z)$, after extending the sequences by zero to
negative indices, gives
\[
\left\|
\bigl((1+n)^{\sigma+\frac1q}B_p(F;n)\bigr)^s
\right\|_{\ell^{\frac{q}{s}}}
\lesssim
\left\|
\bigl((1+m)^\nu c_m\bigr)^s
\right\|_{\ell^{\frac{q}{s}}}.
\]
Equivalently,
\[
\sum_{n=0}^\infty
(1+n)^{1+\sigma q}B_p(F;n)^q
\lesssim
\sum_{m=0}^\infty
(1+m)^{\nu q}c_m^q.
\]
Since
\(
B_{p,\sigma}(F;n)
\asymp_\sigma
(1+n)^\sigma B_p(F;n),
\)
{\rm(F2)} yields
\[
\|S_\Sigma^\phi c\|_{\mathcal F_{\alpha;\sigma}^{p,q}}^q
\lesssim
\sum_{m=0}^\infty
(1+m)^{\nu q}c_m^q
=
\|c\|_{\ell_{w_\nu}^{p,q}(\Sigma)}^q.
\]

\smallskip
\noindent
Suppose now that $q=\infty$. Then
\(s=\min\{1,p\}\) and \(\nu=\sigma-\frac1p\).
Take
\[
L:=|\sigma|+\frac2s.
\]
Multiplying \eqref{eq:weighted-synthesis-block} by
$(1+n)^{s\sigma}$ and using
\[
\left(\frac{1+n}{1+m}\right)^{s\sigma}
\le
(1+|n-m|)^{s|\sigma|}
\]
gives
\begin{equation}
\label{eq:weighted-synthesis-endpoint}
\bigl((1+n)^\sigma B_p(F;n)\bigr)^s
\lesssim
\sum_{m=0}^\infty
(1+|n-m|)^{-2}
\bigl((1+m)^\nu c_m\bigr)^s.
\end{equation}
Taking suprema, we obtain
\[
\sup_{n\ge0}(1+n)^\sigma B_p(F;n)
\lesssim
\sup_{m\ge0}(1+m)^\nu c_m.
\]
Since
\(
B_{p,\sigma}(F;n)
\asymp_\sigma
(1+n)^\sigma B_p(F;n),
\)
{\rm(F2)} yields
\[
\|S_\Sigma^\phi c\|_{\mathcal F_{\alpha;\sigma}^{p,\infty}}
\lesssim
\|c\|_{\ell_{w_\nu}^{p,\infty}(\Sigma)}.
\]
Thus $S_\Sigma^\phi$ is bounded in all cases.

\smallskip
\noindent
Finally, suppose that
\(
c\in\ell_{w_\nu,0}^{p,\infty}(\Sigma).
\)
Then
\[
\bigl((1+m)^\nu c_m\bigr)^s\longrightarrow0.
\]
Since convolution with the $\ell^1$-kernel
\(
\bigl((1+|j|)^{-2}\bigr)_{j\in\mathbb Z}
\)
preserves $c_0$, \eqref{eq:weighted-synthesis-endpoint} implies
\[
(1+n)^\sigma B_p(F;n)\longrightarrow0.
\]
For $r\in[n,n+1)$,
\[
(1+r)^\sigma
M_p(S_\Sigma^\phi c,r)e^{-\frac{\alpha}{2}r^2}
\lesssim_\sigma
(1+n)^\sigma B_p(F;n),
\]
and therefore
\eqref{eq:weighted-localized-synthesis-vanishing} follows.
\end{proof}

\begin{corollary}
\label{cor:normalized-localized-synthesis}
Let $0<p,q\le\infty$, let $\Sigma\subset\mathbb C$ be separated, and
let $(\phi_\xi)_{\xi\in\Sigma}$ be rapidly Fock localized. Define
\[
S_\Sigma^{\phi,p,q}a
:=
\sum_{\xi\in\Sigma}
a_\xi(1+|\xi|)^{\frac1p-\frac1q}\phi_\xi .
\]
Then
\(
S_\Sigma^{\phi,p,q}:
\ell^{p,q}(\Sigma)
\longrightarrow
\mathcal F_\alpha^{p,q}
\)
is a bounded linear operator, and
\begin{equation}\label{eq:normalized-localized-synthesis}
\|S_\Sigma^{\phi,p,q}a\|_{\mathcal F_\alpha^{p,q}}
\lesssim_{\alpha,p,q}
\|a\|_{\ell^{p,q}(\Sigma)}.
\end{equation}
If $q=\infty$, then
\(
S_\Sigma^{\phi,p,\infty}:
\ell_0^{p,\infty}(\Sigma)
\longrightarrow
f_\alpha^{p,\infty}
\)
is bounded. Moreover, if $q<\infty$, the finite-shell partial sums
converge in $\mathcal F_\alpha^{p,q}$, while for
$a\in\ell_0^{p,\infty}(\Sigma)$ they converge in
$\mathcal F_\alpha^{p,\infty}$.
\end{corollary}

\begin{proof}
Apply Theorem~\ref{thm:weighted-localized-synthesis} with $\sigma=0$.
The coefficient space is then \(\ell_{\mathrm{nat}}^{p,q}(\Sigma)\).
Put
\[
c_\xi
:=
a_\xi(1+|\xi|)^{\frac1p-\frac1q}.
\]
For $\xi\in\Sigma_m$,
\(
1+m\le1+|\xi|<2(1+m),
\)
and therefore
\[
\|c\|_{\ell_{\mathrm{nat}}^{p,q}(\Sigma)}
\asymp_{p,q}
\|a\|_{\ell^{p,q}(\Sigma)}.
\]
Thus \eqref{eq:normalized-localized-synthesis} follows directly from
Theorem~\ref{thm:weighted-localized-synthesis}.

\smallskip
\noindent
If $q=\infty$, then $\nu(p,\infty)=-\frac1p$, and the same comparison gives
\[
a\in\ell_0^{p,\infty}(\Sigma)
\quad\Longleftrightarrow\quad
c\in \ell_{\mathrm{nat},0}^{p,\infty}(\Sigma).
\]
Hence the endpoint assertion of
Theorem~\ref{thm:weighted-localized-synthesis} yields
\[
S_\Sigma^{\phi,p,\infty}a\in f_\alpha^{p,\infty},
\]
and the corresponding operator is bounded.

\smallskip
\noindent
Finally, finite-shell truncations converge in the coefficient space
when $q<\infty$, and also for
$a\in\ell_0^{p,\infty}(\Sigma)$ when $q=\infty$. The boundedness of
the synthesis operator therefore gives the asserted finite-shell
convergence in the corresponding Fock space.
\end{proof}

\section{Local geometry and sampling reconstruction}
\label{sec:local-geometry}

This section develops the basic sampling and interpolation tools
used later. We first establish upper restriction and the associated
separation consequences, then extract separated sampling carriers,
and finally extend the canonical reconstruction of a Hilbert
sampling set to the full mixed-norm scale.

\subsection{Upper restriction and local geometry}
\label{subsec:upper-restriction-local-geometry}

\begin{proposition}
\label{prop:upper-restriction}
Let $0<p,q\le\infty$, and let $\Lambda\subset\mathbb C$ be locally
finite. Assume that $\Lambda$ is relatively separated whenever
$p<\infty$. Then
\[
R_{\alpha,\Lambda}^{p,q}:
\mathcal F_\alpha^{p,q}\longrightarrow\ell^{p,q}(\Lambda)
\]
is bounded. If $q=\infty$, then its restriction to
$f_\alpha^{p,\infty}$ defines a bounded map
\[
R_{\alpha,\Lambda}^{p,\infty}:
f_\alpha^{p,\infty}\longrightarrow\ell_0^{p,\infty}(\Lambda).
\]
\end{proposition}

\begin{proof}
For simplicity, set
\[
\nu:=\nu(p,q)=\frac1q-\frac1p.
\]
Assume first that $0<p<\infty$.
For $k\ge0$, set
\[
S_k(f)
:=
\left\|
\bigl(
f(\lambda)e^{-\frac{\alpha}{2}|\lambda|^2}
\bigr)_{\lambda\in\Lambda_k}
\right\|_{\ell^p(\Lambda_k)}.
\]
By the definition of $R_{\alpha,\Lambda}^{p,q}$ and the shell
comparability $1+|\lambda|\asymp1+k$ on $\Lambda_k$,
\begin{equation}
\label{eq:restriction-shell-blocks}
\|R_{\alpha,\Lambda}^{p,q}f\|_{\ell^{p,q}(\Lambda)}
\asymp_{p,q}
\left\|
\bigl((1+k)^{\nu} S_k(f)\bigr)_{k\ge0}
\right\|_{\ell^q}.
\end{equation}
By \cite[Lemma~2.32]{Zhu2012}, the submean inequality applied to
\[
z\longmapsto
f(z)e^{-\alpha z\overline{\lambda}
+\frac{\alpha}{2}|\lambda|^2}
\]
gives
\[
|f(\lambda)|^p e^{-\frac{p\alpha}{2}|\lambda|^2}
\lesssim_{\alpha,p}
\int_{B(\lambda,1)}
|f(z)|^p e^{-\frac{p\alpha}{2}|z|^2}\,dA(z).
\]
Since $\Lambda$ is relatively separated, the disks
$B(\lambda,1)$, $\lambda\in\Lambda$, have uniformly bounded
overlap. Moreover, if $\lambda\in A_k$, then $B(\lambda,1)$
meets only a fixed number of annuli neighboring $A_k$. Hence
\[
S_k(f)^p
\lesssim
\sum_{\substack{j\ge0\\ |j-k|\le2}}
\int_{A_j}
|f(z)|^p e^{-\frac{p\alpha}{2}|z|^2}\,dA(z),
\]
where the implicit constant is independent of $k$ and $f$.
By polar coordinates,
\[
\int_{A_j}
|f(z)|^p e^{-\frac{p\alpha}{2}|z|^2}\,dA(z)
\lesssim
(1+j)B_p(f;j)^p.
\]
Multiplying by $(1+k)^{\nu p}$ and using
$1+k\asymp1+j$ when $|j-k|\le2$, together with
\(
\nu p+1=\frac pq,
\)
we obtain
\begin{equation}
\label{eq:upper-restriction-block}
\bigl((1+k)^\nu S_k(f)\bigr)^p
\lesssim
\sum_{\substack{j\ge0\\ |j-k|\le2}}
\left[(1+j)^{\frac1q}B_p(f;j)\right]^p.
\end{equation}
Suppose first that $q<\infty$. If $q\le p$, then $\frac{q}{p}\le1$, and
subadditivity gives
\[
\left(
\sum_{\substack{j\ge0\\ |j-k|\le2}} \left[(1+j)^{\frac1q}B_p(f;j)\right]^p
\right)^{\frac{q}{p}}
\le
\sum_{\substack{j\ge0\\ |j-k|\le2}} (1+j)B_p(f;j)^q.
\]
If $p<q$, then the equivalence of the $\ell^p$- and $\ell^q$-norms on the
uniformly bounded index window $\{j\ge0:|j-k|\le2\}$
gives the same conclusion. Hence, in either
case,
\[
\bigl((1+k)^\nu S_k(f)\bigr)^q
\lesssim
\sum_{\substack{j\ge0\\ |j-k|\le2}} (1+j)B_p(f;j)^q.
\]
Summing over $k$ and using the finite overlap of the index windows,
we obtain
\[
\left\|
\bigl((1+k)^\nu S_k(f)\bigr)_{k\ge0}
\right\|_{\ell^q}
\lesssim
\left\|\left((1+j)^{\frac{1}{q}}B_p(f;j)\right)_{j\ge0}\right\|_{\ell^q}.
\]
By {\rm(F2)}, the right-hand side is comparable to
$\|f\|_{\mathcal F_\alpha^{p,q}}$.

\smallskip
\noindent
If $q=\infty$, then $\nu=-\frac{1}{p}$, and the same block estimate gives
\[
(1+k)^{-\frac{1}{p}}S_k(f)
\lesssim
\max_{\substack{j\ge0\\ |j-k|\le2}} B_p(f;j).
\]
Taking suprema and using {\rm(F2)} yields the required estimate.

\smallskip
\noindent
Suppose now that $p=\infty$. With
\[
S_k(f)
:=
\sup_{\lambda\in\Lambda_k}
|f(\lambda)|e^{-\frac{\alpha}{2}|\lambda|^2},
\]
we have directly
\[
S_k(f)\le B_\infty(f;k).
\]
Since now $\nu=\frac{1}{q}$, \eqref{eq:restriction-shell-blocks} and
{\rm(F2)} give
\[
\|R_{\alpha,\Lambda}^{\infty,q}f\|_{\ell^{\infty,q}(\Lambda)}
\lesssim
\|f\|_{\mathcal F_\alpha^{\infty,q}}.
\]
No multiplicity estimate is required.

\smallskip
\noindent
Finally, let $f\in f_\alpha^{p,\infty}$. Then
\(
B_p(f;k)\longrightarrow0.
\)
If $0<p<\infty$, taking $q=\infty$ in
\eqref{eq:upper-restriction-block} gives
\[
\bigl((1+k)^{-\frac{1}{p}}S_k(f)\bigr)^p
\lesssim
\sum_{\substack{j\ge0\\ |j-k|\le2}}
B_p(f;j)^p,
\]
and hence
\[
(1+k)^{-\frac{1}{p}}S_k(f)\longrightarrow0.
\]
Since $1+|\lambda|\asymp1+k$ on $\Lambda_k$,
\[
\left\|
(R_{\alpha,\Lambda}^{p,\infty}f)|_{\Lambda_k}
\right\|_{\ell^p(\Lambda_k)}
\asymp_p
(1+k)^{-\frac{1}{p}}S_k(f),
\]
and hence the left-hand side tends to zero.

\smallskip
\noindent
If $p=\infty$, then
\(
S_k(f)\le B_\infty(f;k)\longrightarrow0,
\)
and the same conclusion is immediate. Therefore
\[
R_{\alpha,\Lambda}^{p,\infty}f
\in\ell_0^{p,\infty}(\Lambda).
\]
\end{proof}

The preceding result gives the upper restriction estimate under the
appropriate local geometric hypothesis. We next record the converse geometric implications needed later.

\begin{proposition}
\label{prop:local-geometric-consequences}
Let $0<p,q\le\infty$, and let $\Lambda\subset\mathbb C$ be locally
finite.

\begin{enumerate}
\item[\textnormal{(a)}]
If $0<p<\infty$ and
\(
R_{\alpha,\Lambda}^{p,q}:
\mathcal F_\alpha^{p,q}\longrightarrow\ell^{p,q}(\Lambda)
\)
is bounded, then $\Lambda$ is relatively separated.

\item[\textnormal{(b)}]
If $\Lambda$ is interpolating for $\mathcal F_\alpha^{p,q}$, then
$\Lambda$ is separated. The same conclusion holds if $\Lambda$ is
interpolating for $f_\alpha^{p,\infty}$.
\end{enumerate}
\end{proposition}

\begin{proof}
For simplicity, set
\[
\nu:=\nu(p,q)=\frac1q-\frac1p,
\qquad
g_w:=(1+|w|)^{-\nu}\kappa_{\alpha,w}.
\]
By {\rm(F3)},
\(
\|g_w\|_{\mathcal F_\alpha^{p,q}}\asymp_{\alpha,p,q}1
\)
uniformly in $w\in\mathbb C$.

\smallskip
\noindent
For (a), let
\(
E_w=\Lambda\cap B(w,1).
\)
If $\lambda\in E_w$, then
\[
|\kappa_{\alpha,w}(\lambda)|
e^{-\frac{\alpha}{2}|\lambda|^2}
=
e^{-\frac{\alpha}{2}|\lambda-w|^2}
\ge e^{-\frac{\alpha}{2}}.
\]
Since $1+|\lambda|\asymp1+|w|$ on $B(w,1)$, it follows that
\[
|g_w(\lambda)|e^{-\frac{\alpha}{2}|\lambda|^2}
(1+|\lambda|)^{\nu}
\gtrsim1,
\]
with an implicit constant independent of $w$ and $\lambda$.

\smallskip
\noindent
The disk $B(w,1)$ meets at most a fixed number of the annuli $A_k$.
Hence, for some $k=k(w)$,
\[
\#(E_w\cap A_k)\gtrsim \#E_w.
\]
Since every normalized sample corresponding to
$\lambda\in E_w$ is bounded below by a fixed positive constant,
\[
\|R_{\alpha,\Lambda}^{p,q}g_w\|_{\ell^{p,q}(\Lambda)} \geq \left\|
(R_{\alpha,\Lambda}^{p,q}g_w)|_{\Lambda_k}
\right\|_{\ell^p}
\gtrsim
(\#E_w)^{\frac{1}{p}}.
\]
By the assumed boundedness of the restriction map,
\[
\bigl(\#E_w\bigr)^{\frac{1}{p}}
\lesssim
\|g_w\|_{\mathcal F_\alpha^{p,q}}
\lesssim1
\]
uniformly in $w$. Thus
\[
\sup_{w\in\mathbb C}
\#\bigl(\Lambda\cap B(w,1)\bigr)<\infty,
\]
so $\Lambda$ is relatively separated.

\smallskip
\noindent
For (b), suppose first that $\Lambda$ is interpolating for
$\mathcal F_\alpha^{p,q}$. By the quasi-Banach open mapping theorem,
there is a constant $C>0$ such that every
$a\in\ell^{p,q}(\Lambda)$ admits an interpolant
$f\in\mathcal F_\alpha^{p,q}$ satisfying
\[
R_{\alpha,\Lambda}^{p,q}f=a,
\qquad
\|f\|_{\mathcal F_\alpha^{p,q}}
\le C\|a\|_{\ell^{p,q}(\Lambda)}.
\]
Fix distinct $\lambda,\mu\in\Lambda$ and take the one-point datum
$a=e_\lambda$, for which
\(
\|e_\lambda\|_{\ell^{p,q}(\Lambda)}=1.
\)
Then the corresponding interpolant satisfies
\[
f(\lambda)e^{-\frac{\alpha}{2}|\lambda|^2}
=(1+|\lambda|)^{-\nu},
\qquad
f(\mu)=0, \qquad \text{and} \qquad
\|f\|_{\mathcal F_\alpha^{p,q}}\le C.
\]
Set
\[
f_\lambda^\#(z)
:=
f(z)e^{-\alpha z\overline{\lambda}
+\frac{\alpha}{2}|\lambda|^2}.
\]
If $|\lambda-\mu|\ge1$, there is nothing to prove. Otherwise,
the line segment joining $\lambda$ and $\mu$ lies in $B(\lambda,1)$,
and the fundamental theorem of calculus gives
\[
(1+|\lambda|)^{-\nu}
=
|f_\lambda^\#(\lambda)-f_\lambda^\#(\mu)|
\le
|\lambda-\mu|
\sup_{z\in B(\lambda,1)}
|(f_\lambda^\#)'(z)|.
\]
By Cauchy's estimate and the gauge identity,
\[
\sup_{z\in B(\lambda,1)}
|(f_\lambda^\#)'(z)|
\lesssim
\sup_{z\in B(\lambda,2)}
|f_\lambda^\#(z)|
\lesssim_\alpha
\sup_{z\in B(\lambda,2)}
|f(z)|e^{-\frac{\alpha}{2}|z|^2}.
\]
The point-evaluation estimate \eqref{eq:intro-point-evaluation} and the comparability
$1+|z|\asymp1+|\lambda|$ on $B(\lambda,2)$ therefore give
\[
\sup_{z\in B(\lambda,1)}
|(f_\lambda^\#)'(z)|
\lesssim
(1+|\lambda|)^{-\nu},
\]
with a constant independent of \(\lambda\) and \(\mu\). Consequently,
\[
|\lambda-\mu|\gtrsim1,
\]
uniformly for distinct $\lambda,\mu\in\Lambda$. Hence $\Lambda$ is
separated.

\smallskip
\noindent
If instead $\Lambda$ is interpolating for
$f_\alpha^{p,\infty}$, then
$e_\lambda\in\ell_0^{p,\infty}(\Lambda)$ and
\(
\|e_\lambda\|_{\ell^{p,\infty}(\Lambda)}=1.
\)
The same open-mapping argument gives an interpolant
\(
f\in f_\alpha^{p,\infty}\subset\mathcal F_\alpha^{p,\infty}
\)
with uniformly bounded norm. Repeating the preceding argument  with
$q=\infty$, so that
\(
\nu(p, \infty)=-\frac1p
\)
gives the same uniform lower bound for $|\lambda-\mu|$.
\end{proof}

\subsection{Separated sampling carriers}
\label{subsec:separated-sampling-carriers}

We next show that a sampling configuration contains a separated
subconfiguration that retains the lower sampling estimate. The main
analytic input is the following local oscillation estimate.

\begin{lemma}
\label{lem:local-oscillation}
Let $0<p,q\le\infty$.
Then, for every \(f\in H(\mathbb C)\), the following estimates hold whenever \(|z-w|<\delta<1\).

\smallskip
\noindent
If $0<p<\infty$, then
\begin{align*}
|f(z)|e^{-\frac{\alpha}{2}|z|^2}(1+|z|)^{\nu(p,q)}
&\lesssim_{\alpha,p,q}
\, |f(w)|e^{-\frac{\alpha}{2}|w|^2}(1+|w|)^{\nu(p,q)}\\
&\qquad\quad
+\delta(1+|w|)^{\nu(p,q)}
\left(
\int_{B(w,2)}
|f(\zeta)|^p e^{-\frac{p\alpha}{2}|\zeta|^2}\,dA(\zeta)
\right)^{\frac{1}{p}}.
\end{align*}
If $p=\infty$, then
\begin{align*}
|f(z)|e^{-\frac{\alpha}{2}|z|^2}(1+|z|)^{\nu(p,q)}
&\lesssim_{\alpha,q}
\, |f(w)|e^{-\frac{\alpha}{2}|w|^2}(1+|w|)^{\nu(p,q)}\\
&\quad
+\delta(1+|w|)^{\nu(p,q)}
\sup_{\zeta\in B(w,2)}
|f(\zeta)|e^{-\frac{\alpha}{2}|\zeta|^2}.
\end{align*}
\end{lemma}

\begin{proof}
For simplicity, set
\[
\nu:=\nu(p,q)=\frac1q-\frac1p, \qquad 
f_w^\#(z)
:=
f(z)e^{-\alpha z\overline w+\frac{\alpha}{2}|w|^2}.
\]
Then
\[
|f_w^\#(z)|
=
|f(z)|e^{-\frac{\alpha}{2}|z|^2}
e^{\frac{\alpha}{2}|z-w|^2}.
\]
The fundamental theorem of calculus gives
\[
|f_w^\#(z)|
\le
|f_w^\#(w)|
+
\delta
\sup_{\zeta\in B(w,1)}
|(f_w^\#)'(\zeta)|.
\]
If $0<p<\infty$, Cauchy's estimate followed by the submean
inequality for the subharmonic function $|f_w^\#|^p$, applied on
fixed-radius disks contained in $B(w,2)$, yields
\[
\sup_{\zeta\in B(w,1)}
|(f_w^\#)'(\zeta)|
\lesssim
\left(
\int_{B(w,2)}
|f_w^\#(\zeta)|^p\,dA(\zeta)
\right)^{\frac{1}{p}}.
\]
Since $|\zeta-w|\le2$ on $B(w,2)$,
\[
\sup_{\zeta\in B(w,1)}
|(f_w^\#)'(\zeta)|
\lesssim_{\alpha,p}
\left(
\int_{B(w,2)}
|f(\zeta)|^p e^{-\frac{p\alpha}{2}|\zeta|^2}\,dA(\zeta)
\right)^{\frac{1}{p}}.
\]
For $p=\infty$, Cauchy's estimate and the gauge identity give instead
\[
\sup_{\zeta\in B(w,1)}
|(f_w^\#)'(\zeta)|
\lesssim_{\alpha}
\sup_{\zeta\in B(w,2)}
|f(\zeta)|e^{-\frac{\alpha}{2}|\zeta|^2}.
\]
Finally, we have
\[
|f_w^\#(w)|
=
|f(w)|e^{-\frac{\alpha}{2}|w|^2},
\qquad
|f(z)|e^{-\frac{\alpha}{2}|z|^2}
\le |f_w^\#(z)|, \qquad \text{and} \qquad
1+|z|\asymp1+|w|
\]
whenever $|z-w|<1$. This proves the two estimates.
\end{proof}

\medskip
\noindent
Let $\Lambda\subset\mathbb C$ be locally finite. For $0<\delta<1$,
choose a maximal $\delta$-separated subset
\[
\Gamma_\delta\subset\Lambda.
\]
Since $\Lambda$ is countable, such a set may be obtained by a greedy
enumeration, retaining successively each point whose distance from all
previously retained points is at least $\delta$. By maximality, for
every $\lambda\in\Lambda$ there is
$\gamma(\lambda)\in\Gamma_\delta$ such that
\begin{equation}
\label{eq:carrier-assignment}
|\lambda-\gamma(\lambda)|<\delta.
\end{equation}
When $\Lambda$ is relatively separated, we write
\[
N_\Lambda(R)
:=
\sup_{z\in\mathbb C}\#\bigl(\Lambda\cap B(z,R)\bigr),
\qquad R>0.
\]
Thus $N_\Lambda(R)<\infty$ for every fixed $R>0$.

\begin{proposition}
\label{prop:carrier-compression}
Let $0<p,q\le\infty$, and let $\Lambda\subset\mathbb C$ be locally
finite. Assume that $\Lambda$ is relatively separated whenever
$p<\infty$. Set
\(
s:=\min\{1,p,q\}.
\)
Then, for every $0<\delta<1$, every maximal $\delta$-separated
subset $\Gamma_\delta\subset\Lambda$, and every
$f\in\mathcal F_\alpha^{p,q}$, 
\begin{equation}
\label{eq:carrier-compression}
\|R_{\alpha,\Lambda}^{p,q}f\|_{\ell^{p,q}(\Lambda)}^s
\lesssim_{\alpha,p,q}
\|R_{\alpha,\Gamma_\delta}^{p,q}f\|_
{\ell^{p,q}(\Gamma_\delta)}^s
+
\delta^s\|f\|_{\mathcal F_\alpha^{p,q}}^s,
\end{equation}
where the implicit constant is independent of $\delta$ and $f$.
\end{proposition}

\begin{proof}
For simplicity, set
\[
\nu:=\nu(p,q)=\frac1q-\frac1p.
\]
We first record the effect of the assignment
\eqref{eq:carrier-assignment}. For every
$b=(b_\gamma)_{\gamma\in\Gamma_\delta}$,
\begin{equation}
\label{eq:carrier-assignment-norm}
\left\|
\bigl(b_{\gamma(\lambda)}\bigr)_{\lambda\in\Lambda}
\right\|_{\ell^{p,q}(\Lambda)}
\lesssim_{p,q}
\|b\|_{\ell^{p,q}(\Gamma_\delta)},
\end{equation}
with an implicit constant independent of $\delta$.
Indeed, if $0<p<\infty$ and $\gamma(\lambda)=\gamma$, then
$\lambda\in\Lambda\cap B(\gamma,1)$, and hence
\[
\#\{\lambda\in\Lambda:\gamma(\lambda)=\gamma\}
\le N_\Lambda(1).
\]
Moreover, $\lambda\in A_k$ implies
$\gamma(\lambda)\in A_j$ for some $|j-k|\le1$. Therefore
\[
\left\|
\bigl(b_{\gamma(\lambda)}\bigr)_{\lambda\in\Lambda_k}
\right\|_{\ell^p}^p
=
\sum_{\lambda\in\Lambda_k}|b_{\gamma(\lambda)}|^p\le
N_\Lambda(1)
\sum_{\substack{j\ge0\\ |j-k|\le1}}
\sum_{\gamma\in\Gamma_\delta\cap A_j}|b_\gamma|^p.
\]
The same finite-window argument used in the proof of
Proposition~\ref{prop:upper-restriction}, with the window
$|j-k|\le1$, gives \eqref{eq:carrier-assignment-norm}. If
$p=\infty$, then
\[
\left\|
\bigl(b_{\gamma(\lambda)}\bigr)_{\lambda\in\Lambda_k}
\right\|_{\ell^\infty}
\le
\max_{\substack{j\ge0\\ |j-k|\le1}}
\|b|_{\Gamma_\delta\cap A_j}\|_{\ell^\infty},
\]
and the corresponding finite-window estimate gives the same
conclusion.

\smallskip
\noindent
We now apply Lemma~\ref{lem:local-oscillation} with
\(
w=\gamma(\lambda)\) and \(z=\lambda\).
For $0<p<\infty$, set
\[
Q_\gamma(f)
:=
(1+|\gamma|)^\nu
\left(
\int_{B(\gamma,2)}
|f(\zeta)|^p
e^{-\frac{p\alpha}{2}|\zeta|^2}\,dA(\zeta)
\right)^{\frac{1}{p}},
\]
whereas for $p=\infty$ set
\[
Q_\gamma(f)
:=
(1+|\gamma|)^\nu
\sup_{\zeta\in B(\gamma,2)}
|f(\zeta)|e^{-\frac{\alpha}{2}|\zeta|^2}.
\]
The local oscillation
estimate gives
\[
|(R_{\alpha,\Lambda}^{p,q}f)_\lambda|
\lesssim_{\alpha,p,q}
|(R_{\alpha,\Gamma_\delta}^{p,q}f)_{\gamma(\lambda)}|
+
\delta Q_{\gamma(\lambda)}(f).
\]
By solidity, the powered triangle inequality, and
\eqref{eq:carrier-assignment-norm},
\begin{equation} \label{eq:compression-before-error}
\|R_{\alpha,\Lambda}^{p,q}f\|_{\ell^{p,q}(\Lambda)}^s
\lesssim
\|R_{\alpha,\Gamma_\delta}^{p,q}f\|_
{\ell^{p,q}(\Gamma_\delta)}^s
+
\delta^s
\|Q(f)\|_{\ell^{p,q}(\Gamma_\delta)}^s.
\end{equation}
It remains to prove
\begin{equation}
\label{eq:carrier-error-bound}
\|Q(f)\|_{\ell^{p,q}(\Gamma_\delta)}
\lesssim
\|f\|_{\mathcal F_\alpha^{p,q}},
\end{equation}
with a constant independent of $\delta$.

\smallskip
\noindent
Assume first that $0<p<\infty$.
Since $\Gamma_\delta\subset\Lambda$,
\[
\sup_{\zeta\in\mathbb C}
\sum_{\gamma\in\Gamma_\delta}
\mathbf 1_{B(\gamma,2)}(\zeta)
\le N_\Lambda(2),
\]
so the overlap constant is independent of $\delta$. Using also
$1+|\gamma|\asymp1+k$ for $\gamma\in A_k$, we obtain
\[
\|Q(f)|_{\Gamma_\delta\cap A_k}\|_{\ell^p}^p
\lesssim
(1+k)^{\nu p}
\sum_{\substack{j\ge0\\ |j-k|\le3}}
\int_{A_j}
|f(\zeta)|^p
e^{-\frac{p\alpha}{2}|\zeta|^2}\,dA(\zeta).
\]
As in the proof of Proposition~\ref{prop:upper-restriction},
\[
\int_{A_j}
|f(\zeta)|^p
e^{-\frac{p\alpha}{2}|\zeta|^2}\,dA(\zeta)
\lesssim
(1+j)B_p(f;j)^p.
\]
Since
\(
\nu p+1=\frac pq,
\)
it follows that
\[
\|Q(f)|_{\Gamma_\delta\cap A_k}\|_{\ell^p}^p
\lesssim
\sum_{\substack{j\ge0\\ |j-k|\le3}}
\left[(1+j)^{\frac{1}{q}}B_p(f;j)\right]^p.
\]
The same finite-window argument used in the proof of
Proposition~\ref{prop:upper-restriction}, with the window
$|j-k|\le3$ in place of $|j-k|\le2$, followed by {\rm(F2)}, gives \eqref{eq:carrier-error-bound}.

\smallskip
\noindent
If $p=\infty$ and $\gamma\in A_k$, then
\[
Q_\gamma(f)
\lesssim
(1+k)^{\frac{1}{q}}
\max_{\substack{j\ge0\\ |j-k|\le3}}
B_\infty(f;j).
\]
The same finite-window estimate, followed by {\rm(F2)}, gives
\eqref{eq:carrier-error-bound}.

\smallskip
\noindent
Combining \eqref{eq:compression-before-error} and
\eqref{eq:carrier-error-bound} proves
\eqref{eq:carrier-compression}.
\end{proof}

\begin{proposition}
\label{prop:separated-sampling-carrier}
Let $0<p,q\le\infty$. If $\Lambda$ is sampling for
$\mathcal F_\alpha^{p,q}$, then there exist $0<\delta<1$ and a
$\delta$-separated subset $\Gamma_\delta\subset\Lambda$ such that
\begin{equation}
\label{eq:carrier-lower-sampling}
\|f\|_{\mathcal F_\alpha^{p,q}}
\lesssim_{\alpha,p,q}
\|R_{\alpha,\Gamma_\delta}^{p,q}f\|_
{\ell^{p,q}(\Gamma_\delta)},
\qquad
f\in\mathcal F_\alpha^{p,q}.
\end{equation}
\end{proposition}

\begin{proof}
Let $A>0$ be a lower sampling constant, so that
\[
A\|f\|_{\mathcal F_\alpha^{p,q}}
\le
\|R_{\alpha,\Lambda}^{p,q}f\|_{\ell^{p,q}(\Lambda)},
\qquad
f\in\mathcal F_\alpha^{p,q}.
\]
If $p<\infty$, the upper sampling inequality and
Proposition~\ref{prop:local-geometric-consequences} imply that
$\Lambda$ is relatively separated. Hence
Proposition~\ref{prop:carrier-compression} applies.

\smallskip
\noindent
Set
\[
s:=\min\{1,p,q\},
\]
and let $\Gamma_\delta\subset\Lambda$ be a maximal
$\delta$-separated subset. By
Proposition~\ref{prop:carrier-compression}, there is a constant
$C>0$, independent of $\delta$, such that
\[
A^s\|f\|_{\mathcal F_\alpha^{p,q}}^s
\le
C\|R_{\alpha,\Gamma_\delta}^{p,q}f\|_
{\ell^{p,q}(\Gamma_\delta)}^s
+
C\delta^s\|f\|_{\mathcal F_\alpha^{p,q}}^s.
\]
Choose $0<\delta<1$ sufficiently small that
\[
C\delta^s\le\frac{A^s}{2}.
\]
Absorbing the last term into the left-hand side gives
\eqref{eq:carrier-lower-sampling}.
\end{proof}

\subsection{Reconstruction from Hilbert sampling}
\label{subsec:hilbert-reconstruction-transfer}

We conclude this section by extending the canonical reconstruction
associated with a Hilbert sampling set to the full mixed-norm scale.
The argument combines the Hilbert frame expansion from
Subsection~\ref{SS:HilbertFock}, the normalized localized synthesis
estimate from Subsection~\ref{subsec:localized-synthesis}, and the upper
restriction estimate proved above.

\begin{proposition}
\label{prop:hilbert-reconstruction-transfer}
Let $0<p,q\le\infty$, and let $\Gamma\subset\mathbb C$ be separated
and sampling for $\mathcal F_\alpha^2$. Let
$(\varphi_\gamma)_{\gamma\in\Gamma}$ be its canonical Hilbert dual
family from Proposition~\ref{prop:localized-hilbert-atoms}. Then
\begin{equation}
\label{eq:mixed-hilbert-reconstruction}
f
=
S_\Gamma^{\varphi,p,q}
R_{\alpha,\Gamma}^{p,q}f
=
\sum_{\gamma\in\Gamma}
f(\gamma)e^{-\frac{\alpha}{2}|\gamma|^2}\varphi_\gamma,
\qquad
f\in\mathcal F_\alpha^{p,q}.
\end{equation}
Consequently,
\begin{equation}
\label{eq:hilbert-reconstruction-lower}
\|f\|_{\mathcal F_\alpha^{p,q}}
\lesssim_{\alpha,p,q}
\|R_{\alpha,\Gamma}^{p,q}f\|_{\ell^{p,q}(\Gamma)}.
\end{equation}
\end{proposition}

\begin{proof}
By Proposition~\ref{prop:localized-hilbert-atoms},
$(\varphi_\gamma)_{\gamma\in\Gamma}$ is rapidly Fock localized.
Since $\Gamma$ is separated,
Proposition~\ref{prop:upper-restriction} gives
\[
R_{\alpha,\Gamma}^{p,q}f\in\ell^{p,q}(\Gamma),
\]
and Corollary~\ref{cor:normalized-localized-synthesis} shows that
\[
S_\Gamma^{\varphi,p,q}
R_{\alpha,\Gamma}^{p,q}f
\in\mathcal F_\alpha^{p,q}.
\]
Its defining series is precisely the series in
\eqref{eq:mixed-hilbert-reconstruction}. It remains to identify its
sum with $f$.

\smallskip
\noindent
For $0<\tau<1$, set
\[
f_\tau:=D_\tau f,
\qquad
f_\tau(z)=f(\tau z).
\]
Then $f_\tau\to f$ locally uniformly as $\tau\uparrow1$, and
$f_\tau\in\mathcal F_\alpha^2$. Indeed, the point-evaluation estimate
\eqref{eq:intro-point-evaluation} gives
\[
|f_\tau(z)|^2e^{-\alpha|z|^2}
\lesssim_{\alpha,p,q}
(1+|z|)^{2\left|\frac1p-\frac1q\right|}
e^{-\alpha(1-\tau^2)|z|^2}
\|f\|_{\mathcal F_\alpha^{p,q}}^2,
\]
which is integrable for every fixed $\tau<1$.

\smallskip
\noindent
By Proposition~\ref{prop:localized-hilbert-atoms}(i),
\begin{equation}
\label{eq:dilated-hilbert-reconstruction}
f_\tau
=
\sum_{\gamma\in\Gamma}
f_\tau(\gamma)e^{-\frac{\alpha}{2}|\gamma|^2}\varphi_\gamma .
\end{equation}
The series converges in $\mathcal F_\alpha^2$, hence locally
uniformly.

\smallskip
\noindent
Suppose first that $0<q<\infty$. 
By {\rm(F4)} and Proposition~\ref{prop:upper-restriction},
\[
R_{\alpha,\Gamma}^{p,q}f_\tau
\longrightarrow
R_{\alpha,\Gamma}^{p,q}f
\qquad\text{in }\ell^{p,q}(\Gamma).
\]
Since $S_\Gamma^{\varphi,p,q}$ is bounded,
\[
S_\Gamma^{\varphi,p,q}R_{\alpha,\Gamma}^{p,q}f_\tau
\longrightarrow
S_\Gamma^{\varphi,p,q}R_{\alpha,\Gamma}^{p,q}f
\]
in $\mathcal F_\alpha^{p,q}$. The left-hand side equals $f_\tau$ by
\eqref{eq:dilated-hilbert-reconstruction}, while $f_\tau\to f$ in
$\mathcal F_\alpha^{p,q}$. Hence
\[
f=S_\Gamma^{\varphi,p,q}R_{\alpha,\Gamma}^{p,q}f,
\]
which proves \eqref{eq:mixed-hilbert-reconstruction}.

\smallskip
\noindent
Suppose now that $q=\infty$. For $t\ge0$,
\[
M_p(f_\tau,t)e^{-\frac{\alpha}{2}t^2}
=
M_p(f,\tau t)e^{-\frac{\alpha}{2}\tau^2 t^2}
e^{-\frac{\alpha}{2}(1-\tau^2)t^2} \le
\|f\|_{\mathcal F_\alpha^{p,\infty}},
\]
and therefore
\begin{equation}
\label{eq:dilation-uniform-endpoint}
\sup_{0<\tau<1}
\|f_\tau\|_{\mathcal F_\alpha^{p,\infty}}
\le
\|f\|_{\mathcal F_\alpha^{p,\infty}}.
\end{equation}
Choose $\tau_j\uparrow1$ and set
\[
a^{(j)}
:=
R_{\alpha,\Gamma}^{p,\infty}f_{\tau_j},
\qquad
a
:=
R_{\alpha,\Gamma}^{p,\infty}f.
\]
By Proposition~\ref{prop:upper-restriction} and
\eqref{eq:dilation-uniform-endpoint},
\[
\sup_j
\|a^{(j)}\|_{\ell^{p,\infty}(\Gamma)}
<\infty,
\]
while local uniform convergence of $f_{\tau_j}$ gives
\[
a_\gamma^{(j)}\longrightarrow a_\gamma,
\qquad
\gamma\in\Gamma.
\]
We claim that
\begin{equation}
\label{eq:coordinatewise-synthesis-limit}
S_\Gamma^{\varphi,p,\infty}a^{(j)}
\longrightarrow
S_\Gamma^{\varphi,p,\infty}a
\end{equation}
locally uniformly. Indeed, if $K\subset\mathbb C$ is compact, rapid
localization and separated shell counting give, for every $N>0$,
\[
\sup_j\sup_{z\in K}
\sum_{\gamma\in\Gamma_m}
|a_\gamma^{(j)}|
(1+|\gamma|)^{\frac{1}{p}}
|\varphi_\gamma(z)|
\lesssim_{K,N}
(1+m)^{-N+\frac1p+\frac1{p^*}}
\sup_j\|a^{(j)}\|_{\ell^{p,\infty}(\Gamma)};
\]
the same estimate
holds for $a$. Choosing $N$ sufficiently large makes the right-hand
side summable in $m$. Thus the tails are uniformly small on $K$,
while on finitely many shells coordinatewise convergence is
finite-dimensional. This proves
\eqref{eq:coordinatewise-synthesis-limit}.

\smallskip
\noindent
By \eqref{eq:dilated-hilbert-reconstruction}, the left-hand side of
\eqref{eq:coordinatewise-synthesis-limit} equals $f_{\tau_j}$.
Since $f_{\tau_j}\to f$ locally uniformly, we obtain
\[
f=S_\Gamma^{\varphi,p,\infty}a,
\]
which proves \eqref{eq:mixed-hilbert-reconstruction} for
$q=\infty$.

\smallskip
\noindent
Finally, the boundedness of the normalized synthesis operator gives
\[
\|f\|_{\mathcal F_\alpha^{p,q}}
=
\|S_\Gamma^{\varphi,p,q}
R_{\alpha,\Gamma}^{p,q}f\|_{\mathcal F_\alpha^{p,q}}
\lesssim
\|R_{\alpha,\Gamma}^{p,q}f\|_{\ell^{p,q}(\Gamma)},
\]
and \eqref{eq:hilbert-reconstruction-lower} follows.
\end{proof}

\section{Sufficiency}
\label{sec:sufficiency}

We now prove the sufficiency directions of the main theorems.
Interpolation follows by applying the localized synthesis theorem to
the Hilbert Lagrange family, while sampling follows from the mixed
reconstruction formula associated with a separated Hilbert sampling
carrier.

\subsection{Interpolation sufficiency}
\label{subsec:interpolation-sufficiency}

We begin with interpolation.

\begin{theorem}
\label{thm:interpolation-sufficiency}
Let $0<p,q\le\infty$, and let $\Lambda\subset\mathbb C$ be separated.
If
\[
D^+(\Lambda)<\frac{\alpha}{\pi},
\]
then the normalized restriction map
\(
R_{\alpha,\Lambda}^{p,q}:
\mathcal F_\alpha^{p,q}
\longrightarrow
\ell^{p,q}(\Lambda)
\)
admits the synthesis operator
\(
S_\Lambda^{\psi,p,q}:
\ell^{p,q}(\Lambda)
\longrightarrow
\mathcal F_\alpha^{p,q}
\)
as a bounded linear right inverse, where
$(\psi_\lambda)_{\lambda\in\Lambda}$ is the Hilbert Lagrange family
from Proposition~\ref{prop:localized-hilbert-atoms}.

\smallskip
\noindent
If $q=\infty$, then
\(
S_\Lambda^{\psi,p,\infty}:
\ell_0^{p,\infty}(\Lambda)
\longrightarrow
f_\alpha^{p,\infty}
\)
is also a bounded linear right inverse of
$R_{\alpha,\Lambda}^{p,\infty}$ on $\ell_0^{p,\infty}(\Lambda)$.
\end{theorem}

\begin{proof}

Since $\Lambda$ is separated, hence relatively separated,
Proposition~\ref{prop:upper-restriction} gives the required
boundedness of the restriction maps, including the little endpoint.

\smallskip
\noindent
By Theorem~\ref{thm:classical-hilbert-density}, the density assumption
implies that $\Lambda$ is interpolating for $\mathcal F_\alpha^2$.
Hence Proposition~\ref{prop:localized-hilbert-atoms}(ii) provides a
rapidly Fock-localized Lagrange family
$(\psi_\lambda)_{\lambda\in\Lambda}$ satisfying
\[
\psi_\lambda(\mu)e^{-\frac{\alpha}{2}|\mu|^2}
=
\delta_{\lambda,\mu},
\qquad
\lambda,\mu\in\Lambda.
\]
By Corollary~\ref{cor:normalized-localized-synthesis},
\(
S_\Lambda^{\psi,p,q}:
\ell^{p,q}(\Lambda)
\longrightarrow
\mathcal F_\alpha^{p,q}
\)
is bounded. Moreover, for $a\in\ell^{p,q}(\Lambda)$ and
$\mu\in\Lambda$,
\[
(R_{\alpha,\Lambda}^{p,q}
S_\Lambda^{\psi,p,q}a)_\mu
=
(S_\Lambda^{\psi,p,q}a)(\mu)
e^{-\frac{\alpha}{2}|\mu|^2}
(1+|\mu|)^{\frac1q-\frac1p}  =a_\mu,
\]
by the Lagrange identity. Thus $S_\Lambda^{\psi,p,q}$ is a bounded
linear right inverse of $R_{\alpha,\Lambda}^{p,q}$.

\smallskip
\noindent
If $q=\infty$, Corollary~\ref{cor:normalized-localized-synthesis}
also gives
\(
S_\Lambda^{\psi,p,\infty}:
\ell_0^{p,\infty}(\Lambda)
\longrightarrow
f_\alpha^{p,\infty}
\)
bounded. The same Lagrange identity shows that it is a right inverse
of $R_{\alpha,\Lambda}^{p,\infty}$ on
$\ell_0^{p,\infty}(\Lambda)$.
\end{proof}

\begin{remark}
\label{rem:uniform-interpolation}
The conclusion of Theorem~\ref{thm:interpolation-sufficiency} is
stronger than the usual consequence of surjectivity. Indeed, if
\[
R_{\alpha,\Lambda}^{p,q}:
\mathcal F_\alpha^{p,q}\longrightarrow\ell^{p,q}(\Lambda)
\]
is merely bounded and onto, then the quasi-Banach open mapping theorem
yields a constant $C>0$ such that, for every
$a\in\ell^{p,q}(\Lambda)$, one can choose
$f\in\mathcal F_\alpha^{p,q}$ satisfying
\[
R_{\alpha,\Lambda}^{p,q}f=a,
\qquad
\|f\|_{\mathcal F_\alpha^{p,q}}
\le C\|a\|_{\ell^{p,q}(\Lambda)}.
\]
This does not, in general, provide a linear choice of the preimages.
Under the density hypothesis of Theorem~\ref{thm:interpolation-sufficiency},
the Hilbert Lagrange family instead gives the explicit bounded linear
choice
\[
f=S_\Lambda^{\psi,p,q}a.
\]
The same observation applies to
\[
R_{\alpha,\Lambda}^{p,\infty}:
f_\alpha^{p,\infty}\longrightarrow\ell_0^{p,\infty}(\Lambda).
\]
\end{remark}

\subsection{Sampling sufficiency}
\label{subsec:sampling-sufficiency}

We now prove sampling sufficiency. A separated carrier above the
critical density is Hilbert sampling, and Proposition~\ref{prop:hilbert-reconstruction-transfer}
transfers its reconstruction to the mixed-norm scale.

\begin{theorem}
\label{thm:sampling-sufficiency}
Let $0<p,q\le\infty$, and let $\Lambda\subset\mathbb C$ be locally
finite. If $0<p<\infty$, assume in addition that $\Lambda$ is
relatively separated. If
\[
D_{\rm sep}^-(\Lambda)>\frac{\alpha}{\pi},
\]
then $\Lambda$ is sampling for $\mathcal F_\alpha^{p,q}$.
\end{theorem}

\begin{proof}
By the definition of $D_{\rm sep}^-(\Lambda)$, there is a separated
set $\Gamma\subset\Lambda$ such that
\[
D^-(\Gamma)>\frac{\alpha}{\pi}.
\]
By Theorem~\ref{thm:classical-hilbert-density}, $\Gamma$ is sampling for
$\mathcal F_\alpha^2$. Proposition~\ref{prop:hilbert-reconstruction-transfer}
therefore gives
\[
\|f\|_{\mathcal F_\alpha^{p,q}}
\lesssim_{\alpha,p,q}
\|R_{\alpha,\Gamma}^{p,q}f\|_{\ell^{p,q}(\Gamma)}
\le
\|R_{\alpha,\Lambda}^{p,q}f\|_{\ell^{p,q}(\Lambda)},
\]
where the last inequality follows from solidity. Thus $\Lambda$
satisfies the lower sampling inequality.

\smallskip
\noindent
The upper sampling inequality follows from
Proposition~\ref{prop:upper-restriction}; for $p<\infty$ this uses
the assumed relative separation of $\Lambda$, while for $p=\infty$
local finiteness suffices. Hence $\Lambda$ is sampling for
$\mathcal F_\alpha^{p,q}$.
\end{proof}

\section{Localized lower-stability transfer}
\label{sec:localized-stability}

We prove the discrete matrix transfer principle used in both
necessity arguments. Apart from the weighted annular sequence-space
notation introduced in Subsection~\ref{SS:AdditionalNotation}, the argument is independent
of the Fock-space structure developed in Sections~\ref{sec:structural-inputs} and \ref{sec:local-geometry}.
The main step is to pass directly from lower stability on a weighted
annular mixed sequence space to lower stability on $\ell^\infty$;
the latter is then transferred to $\ell^2$ by the $p$-independence
theorem for the Sjöstrand class in \cite[Theorem~2.1]{ShinSun2009}.

\smallskip
\noindent
Throughout this section, $Y\subset\mathbb C$ is relatively separated
and
\[
X=X^{(1)}\sqcup\cdots\sqcup X^{(J)}
\]
is a finite labeled union of relatively separated subsets of
$\mathbb C$. Distances between labeled points refer to their planar
coordinates. For $m,n\in\mathbb N_0$, set
\[
Y_m:=Y\cap A_m,
\qquad
X_n:=X\cap A_n,
\]
where the intersection with $X$ retains the labels.

\smallskip
\noindent
For $\nu\in\mathbb R$, the weighted annular mixed sequence space
$\ell_{w_\nu}^{p,q}(X)$ is defined as in Subsection~\ref{SS:AdditionalNotation}, with the inner
$\ell^p$-norm taken over the labeled block $X_n$.

\subsection{Localized matrices on weighted annular mixed sequence spaces}
\label{subsec:localized-matrices-weighted}
We first record the boundedness consequences of rapid off-diagonal
decay on weighted annular mixed sequence spaces.

\begin{definition}[Rapid matrix localization]
\label{def:rapid-matrix-localization}
A matrix
\(
T=(T_{x,y})_{x\in X,\;y\in Y}
\)
is called \emph{rapidly localized} if, for every $N>0$, there exists
$C_N>0$ such that
\[
|T_{x,y}|
\le
C_N(1+|x-y|)^{-N},
\qquad x\in X,\ y\in Y.
\]
\end{definition}

\begin{lemma}
\label{lem:shell-block-estimate}
Let $T:Y\to X$ be rapidly localized. Then, for every
$0<p\le\infty$ and every $L>0$,
\begin{equation}
\label{eq:shell-block-estimate}
\left\|
\left(
\sum_{y\in Y_m}T_{x,y}c_y
\right)_{x\in X_n}
\right\|_{\ell^p(X_n)}
\lesssim_{p,L}
(1+|n-m|)^{-L}
\|c\|_{\ell^p(Y_m)}
\end{equation}
for all $m,n\in\mathbb N_0$ and every
$c=(c_y)_{y\in Y_m}$.
\end{lemma}

\begin{proof}
If $x\in X_n$ and $y\in Y_m$, then
\[
1+|n-m|\lesssim 1+|x-y|.
\]
Suppose first that $1\le p\le\infty$. Choose the localization order
$N>L+2$. Then
\[
|T_{x,y}|
\lesssim_{N,L}
(1+|n-m|)^{-L}
(1+|x-y|)^{-N+L}.
\]
Since $Y$ is relatively separated and $X$ is a finite labeled union
of relatively separated configurations,
\[
\sup_{x\in X}\sum_{y\in Y}(1+|x-y|)^{-N+L}
+
\sup_{y\in Y}\sum_{x\in X}(1+|x-y|)^{-N+L}
<\infty.
\]
The Schur test therefore gives
\eqref{eq:shell-block-estimate}.

\smallskip
\noindent
If $0<p<1$, then $p$-subadditivity gives
\[
\begin{aligned}
\left\|
\left(
\sum_{y\in Y_m}T_{x,y}c_y
\right)_{x\in X_n}
\right\|_{\ell^p(X_n)}^p
&\le
\sum_{y\in Y_m}|c_y|^p
\sum_{x\in X_n}|T_{x,y}|^p.
\end{aligned}
\]
Choose the localization order $N>L+\frac{2}{p}$. As above,
\[
|T_{x,y}|^p
\lesssim_L
(1+|n-m|)^{-Lp}
(1+|x-y|)^{-(N-L)p},
\]
and \((N-L)p>2\). Relative-separation counting on the finite labeled
union $X$ therefore gives
\[
\sup_{y\in Y}
\sum_{x\in X}
(1+|x-y|)^{-(N-L)p}
<\infty.
\]
Hence
\[
\left\|
\left(
\sum_{y\in Y_m}T_{x,y}c_y
\right)_{x\in X_n}
\right\|_{\ell^p(X_n)}^p
\lesssim_{p,L}
(1+|n-m|)^{-Lp}
\|c\|_{\ell^p(Y_m)}^p,
\]
which proves the claim.
\end{proof}

\begin{lemma}
\label{lem:localized-matrix-boundedness}
Let $T:Y\to X$ be rapidly localized. Then, for every
$0<p,q\le\infty$ and every $\nu\in\mathbb R$, the matrix $T$
defines a bounded operator
\begin{equation*}
\label{eq:localized-weighted-boundedness}
T:
\ell_{w_\nu}^{p,q}(Y)
\longrightarrow
\ell_{w_\nu}^{p,q}(X).
\end{equation*}
It also defines a bounded operator
\begin{equation}
\label{eq:localized-infty-boundedness}
T:\ell^\infty(Y)\longrightarrow\ell^\infty(X),
\end{equation}
and the series defining each $(Tc)_x$ is absolutely convergent for
$c\in\ell^\infty(Y)$.
\end{lemma}

\begin{proof}
Set
\[
s:=\min\{1,p,q\}.
\]
Suppose first that $c$ is supported in finitely many annular blocks.
By Lemma~\ref{lem:shell-block-estimate} and the powered triangle
inequality, for every $L>0$,
\begin{equation*}
\label{eq:shell-block-convolution}
\|(Tc)|_{X_n}\|_{\ell^p(X_n)}^s
\lesssim_{p,L}
\sum_{m=0}^\infty
(1+|n-m|)^{-Ls}\|c|_{Y_m}\|_{\ell^p(Y_m)}^s.
\end{equation*}
Indeed, if $p\ge1$, one first uses the triangle inequality in
$\ell^p(X_n)$ and then $s$-subadditivity of scalar sums. If
$0<p<1$, one first uses $p$-subadditivity and then raises the
resulting estimate to the power $\frac{s}{p}\le1$.

\smallskip
\noindent
Since
\[
\frac{w_\nu(n)}{w_\nu(m)} = \frac{(1 + n)^{\nu}}{(1 + m)^{\nu}}
\le (1+|n-m|)^{|\nu|},
\]
we obtain
\[
w_\nu(n)^s\|(Tc)|_{X_n}\|_{\ell^p(X_n)}^s
\lesssim_L
\sum_{m=0}^\infty
(1+|n-m|)^{-(L-|\nu|)s}w_\nu(m)^s\|c|_{Y_m}\|_{\ell^p(Y_m)}^s.
\]
Choosing $L>|\nu|+\frac{1}{s}$ and applying Young's inequality on
$\ell^{\frac{q}{s}}$, with the usual supremum interpretation when $q=\infty$,
gives
\[
\|Tc\|_{\ell_{w_\nu}^{p,q}(X)}
\lesssim_{p,q,\nu}
\|c\|_{\ell_{w_\nu}^{p,q}(Y)}.
\]

\smallskip
\noindent
Now let $c\in\ell_{w_\nu}^{p,q}(Y)$ be arbitrary, and set
\[
c^{[M]}
:=
c\,\mathbf 1_{Y_0\cup\cdots\cup Y_M}.
\]
Since
\[
(1+m)^\nu
\|c|_{Y_m}\|_{\ell^p(Y_m)}
\le
\|c\|_{\ell_{w_\nu}^{p,q}(Y)},
\]
the block norms $\|c|_{Y_m}\|_{\ell^p(Y_m)}$ have at most
polynomial growth. Hence, by
Lemma~\ref{lem:shell-block-estimate} with sufficiently large $L$,
for each fixed $n$ the series
\[
\sum_{m=0}^\infty
\left(
\sum_{y\in Y_m}T_{x,y}c_y
\right)_{x\in X_n}
\]
converges in $\ell^p(X_n)$; its sum defines $(Tc)|_{X_n}$.
Consequently,
\[
(Tc^{[M]})|_{X_n}
\longrightarrow
(Tc)|_{X_n}
\qquad\text{in }\ell^p(X_n)
\]
for every fixed $n$.

\smallskip
\noindent
The estimate already proved for $c^{[M]}$ is uniform in $M$.
Passing to the blockwise limit and using Fatou's lemma when
$q<\infty$, and the supremum when $q=\infty$, gives
\[
\|Tc\|_{\ell_{w_\nu}^{p,q}(X)}
\lesssim
\|c\|_{\ell_{w_\nu}^{p,q}(Y)}.
\]
Thus no density of finitely supported sequences is required when
$q=\infty$.

\smallskip
\noindent
Finally, choosing $N>2$ in the rapid localization estimate and using
relative-separation counting gives
\[
\sup_{x\in X}
\sum_{y\in Y}|T_{x,y}|<\infty.
\]
Therefore, for every $c\in\ell^\infty(Y)$,
\[
|(Tc)_x|
\le
\|c\|_{\ell^\infty(Y)}
\sum_{y\in Y}|T_{x,y}|,
\]
uniformly in $x\in X$. This proves
\eqref{eq:localized-infty-boundedness} and the absolute convergence
of the defining row sums.
\end{proof}

\subsection{Translated cutoffs and lower-stability transfer to $\ell^\infty$}
\label{subsec:cutoffs-infty-stability}

We next localize bounded sequences near arbitrary centers. The
translated cutoff estimates compensate for the lack of translation
invariance of the annular weights and, together with a commutator
estimate, transfer weighted mixed lower stability to $\ell^\infty$.

\smallskip
\noindent
For $M>0$, $R\ge1$, and $w\in\mathbb C$, define
\[
\theta_{R,w}(\xi)
:=
\left(1+\frac{|\xi-w|}{R}\right)^{-M}.
\]
If $Z$ is a configuration, let $\Theta_{R,w}^Z$ denote
multiplication by $\theta_{R,w}$ on sequences indexed by $Z$.

\begin{lemma}
\label{lem:translated-cutoff-estimates}
Let $Z\subset\mathbb C$ be relatively separated, or a finite labeled
union of relatively separated configurations. Let
$0<p,q\le\infty$, $\nu\in\mathbb R$, and suppose that
\begin{equation}
\label{eq:cutoff-order}
M>\frac2p+|\nu|+\frac1q.
\end{equation}
Then, for every $b\in\ell^\infty(Z)$,
\begin{equation}
\label{eq:translated-cutoff-upper}
\|\Theta_{R,w}^Zb\|_{\ell_{w_\nu}^{p,q}(Z)}
\lesssim_{p,q,\nu,M}
R^{\frac2p+|\nu|+\frac1q}
(1+|w|)^\nu
\|b\|_{\ell^\infty(Z)}
\end{equation}
uniformly in $w\in\mathbb C$ and $R\ge1$.

\smallskip
\noindent
If $q=\infty$, then
\begin{equation}
\label{eq:translated-cutoff-little}
\Theta_{R,w}^Zb
\in
\ell_{w_\nu,0}^{p,\infty}(Z).
\end{equation}
Consequently, in all cases the finite annular truncations of
$\Theta_{R,w}^Zb$ converge to it in
$\ell_{w_\nu}^{p,q}(Z)$.

\smallskip
\noindent
Moreover, if $b\ne0$ and $\xi_0\in Z$ satisfies
\[
|b_{\xi_0}|
\ge
\frac12\|b\|_{\ell^\infty(Z)},
\]
then, for every $R\ge1$,
\begin{equation}
\label{eq:lower-cutoff-bound}
\|\Theta_{R,\xi_0}^Zb\|_{\ell_{w_\nu}^{p,q}(Z)}
\gtrsim_\nu
(1+|\xi_0|)^\nu
\|b\|_{\ell^\infty(Z)}.
\end{equation}
\end{lemma}

\begin{proof}
We first record the counting estimate used below. If $\beta>2$,
relative separation gives, uniformly in $w\in\mathbb C$,
\[
\#\{\xi\in Z:j\le|\xi-w|<j+1\}
\lesssim 1+j.
\]
Hence, by comparison with the corresponding integral,
\begin{equation}
\label{eq:centered-distance-tail}
\sum_{\substack{\xi\in Z\\|\xi-w|\ge d}}
\left(1+\frac{|\xi-w|}{R}\right)^{-\beta}
\lesssim_{\beta}
R^2
\left(1+\frac dR\right)^{-\beta+2},
\end{equation}
uniformly in $w\in\mathbb C$, $d\ge0$, and $R\ge1$. For a finite labeled union, only
the counting constant changes.

\smallskip
\noindent
Put
\[
\qquad
d_k:=\operatorname{dist}(|w|,[k,k+1]).
\]
If $\xi\in Z\cap A_k$, then $|\xi-w|\ge d_k$. Suppose first that
$0<p<\infty$. Since \eqref{eq:cutoff-order} implies $Mp>2$,
\eqref{eq:centered-distance-tail} gives
\[
\|\Theta_{R,w}^Zb|_{Z\cap A_k}\|_{\ell^p}^p
\le
\|b\|_{\ell^\infty}^p
\sum_{\substack{\xi\in Z\\|\xi-w|\ge d_k}}
\left(1+\frac{|\xi-w|}{R}\right)^{-Mp} \lesssim
R^2
\left(1+\frac{d_k}{R}\right)^{-Mp+2}
\|b\|_{\ell^\infty}^p.
\]
For $p=\infty$,
\[
\|\Theta_{R,w}^Zb|_{Z\cap A_k}\|_{\ell^\infty}
\le
\left(1+\frac{d_k}{R}\right)^{-M}
\|b\|_{\ell^\infty}.
\]
Since $R\ge1$,
\[
1+\frac{|k-|w||}{R}
\lesssim
1+\frac{d_k}{R}
\qquad \text{and} \qquad
(1+k)^\nu
\le
(1+|w|)^\nu
(1+|k-|w||)^{|\nu|}.
\]
Hence, with the convention $\frac{2}{p}=0$ when $p=\infty$,
\begin{equation}
\label{eq:cutoff-block-estimate}
w_\nu(k)
\|\Theta_{R,w}^Zb|_{Z\cap A_k}\|_{\ell^p}
\lesssim_{p,\nu,M}
R^{\frac{2}{p}+|\nu|}
(1+|w|)^\nu
\left(1+\frac{|k-|w||}{R}\right)^{-\delta}
\|b\|_{\ell^\infty},
\end{equation}
where
\[
\delta:=M-\frac{2}{p}-|\nu|>\frac1q.
\]

\smallskip
\noindent
If $q<\infty$, then
\[
\sup_{|w|\ge0}
\sum_{k=0}^\infty
\left(1+\frac{|k-|w||}{R}\right)^{-\delta q}
\lesssim_{\delta,q}R.
\]
Taking the outer $\ell^q$-norm in
\eqref{eq:cutoff-block-estimate} proves
\eqref{eq:translated-cutoff-upper}. If $q=\infty$, taking the
supremum over $k$ gives the same estimate, while $\delta>0$ implies
that the right-hand side of \eqref{eq:cutoff-block-estimate} tends
to zero as $k\to\infty$. This proves
\eqref{eq:translated-cutoff-little}. The assertion concerning finite
annular truncations follows immediately.

\smallskip
\noindent
Finally, let $\xi_0\in A_{k_0}$ satisfy the near-maximality condition.
Since
\(
\theta_{R,\xi_0}(\xi_0)=1,
\)
the $k_0$-th block gives
\[
\|\Theta_{R,\xi_0}^Zb\|_{\ell_{w_\nu}^{p,q}(Z)}
\ge
(1+k_0)^\nu |b_{\xi_0}|.
\]
Moreover, 
\[
(1+k_0)^\nu\asymp_\nu(1+|\xi_0|)^\nu,
\]
hence \eqref{eq:lower-cutoff-bound} follows.
\end{proof}

\begin{lemma}
\label{lem:translated-commutator}
Let $T:Y\to X$ be rapidly localized. Let $0<p,q\le\infty$,
$\nu\in\mathbb R$, and suppose that $M$ satisfies
\eqref{eq:cutoff-order}. Then
\begin{equation}
\label{eq:translated-commutator}
\bigl\|
T\Theta_{R,w}^Yb-\Theta_{R,w}^XTb
\bigr\|_{\ell_{w_\nu}^{p,q}(X)}
\lesssim_{p,q,\nu,M}
\frac1R
\|\Theta_{R,w}^Yb\|_{\ell_{w_\nu}^{p,q}(Y)}
\end{equation}
for every $b\in\ell^\infty(Y)$, uniformly in
$w\in\mathbb C$ and $R\ge1$.
\end{lemma}

\begin{proof}
By Lemma~\ref{lem:localized-matrix-boundedness}, $Tb$ is defined by
absolutely convergent row sums. Hence, for $x\in X$,
\[
\bigl(
T\Theta_{R,w}^Yb-\Theta_{R,w}^XTb
\bigr)_x
=
\sum_{y\in Y}
T_{x,y}
\bigl(\theta_{R,w}(y)-\theta_{R,w}(x)\bigr)b_y.
\]

\smallskip
\noindent
We first estimate the cutoff difference.
The mean-value theorem and $||x-w|-|y-w||\le |x-y|$ give
\[
|\theta_{R,w}(y)-\theta_{R,w}(x)|
\le
\frac{M|x-y|}{R}
\left(
1+\frac{\min\{|x-w|, |y - w|\}}{R}
\right)^{-M-1}.
\]
If the minimum is $|y -w|$, the last factor is bounded by
$\theta_{R,w}(y)$. If the minimum is $|x-w|$, then
\[
1+\frac{|y-w|}{R}
\le
\left(1+\frac{|x-w|}{R}\right)
\left(1+\frac{|x-y|}{R}\right),
\]
and therefore
\[
\left(1+\frac{|x-w|}{R}\right)^{-M-1}
\le
\left(1+\frac{|x-y|}{R}\right)^M
\theta_{R,w}(y).
\]
Since $R\ge1$,
\[
|\theta_{R,w}(y)-\theta_{R,w}(x)|
\lesssim_M
\frac1R
(1+|x-y|)^{M+1}\theta_{R,w}(y).
\]

\smallskip
\noindent
It follows that
\[
\left|
\bigl(
T\Theta_{R,w}^Yb-\Theta_{R,w}^XTb
\bigr)_x
\right|
\lesssim
\frac1R
\sum_{y\in Y}
(1+|x-y|)^{M+1}|T_{x,y}|\theta_{R,w}(y)|b_y|.
\]
Since $T$ is rapidly localized, the matrix with entries
\(
(1+|x-y|)^{M+1}|T_{x,y}|
\)
is still rapidly localized. Moreover,
Lemma~\ref{lem:translated-cutoff-estimates} gives
\[
\bigl(\theta_{R,w}(y)|b_y|\bigr)_{y\in Y}
\in
\ell_{w_\nu}^{p,q}(Y).
\]
Hence Lemma~\ref{lem:localized-matrix-boundedness}, together with
solidity, yields \eqref{eq:translated-commutator}.
\end{proof}

\begin{proposition}[Mixed lower stability implies $\ell^\infty$-lower stability]
\label{prop:mixed-to-infty-stability}
Let $T:Y\to X$ be rapidly localized. Fix
$0<p,q\le\infty$ and $\nu\in\mathbb R$, and suppose that there exists $C_0>0$ such that
\begin{equation}
\label{eq:finite-support-lower-stability}
\|c\|_{\ell_{w_\nu}^{p,q}(Y)}
\le
C_0\|Tc\|_{\ell_{w_\nu}^{p,q}(X)},
\qquad c\in c_{00}(Y).
\end{equation}
Then
\begin{equation}
\label{eq:infty-lower-stability}
\|b\|_{\ell^\infty(Y)}
\lesssim_{p,q,\nu,C_0}
\|Tb\|_{\ell^\infty(X)},
\qquad b\in\ell^\infty(Y).
\end{equation}
\end{proposition}

\begin{proof}
Put
\[
s:=\min\{1,p,q\}.
\]
The assertion is trivial for $b=0$. Otherwise, choose $y_0\in Y$
such that
\[
|b_{y_0}|
\ge
\frac12\|b\|_{\ell^\infty(Y)}.
\]

\smallskip
\noindent
We first extend \eqref{eq:finite-support-lower-stability} to \(\Theta_{R,y_0}^Yb\) with \(R \geq 1\). 
Choose $M$ satisfying \eqref{eq:cutoff-order}.
Set
\[
u: = \Theta_{R,y_0}^Yb \qquad \text{and} \qquad u^{[N]}
:=
u\,\mathbf 1_{Y_0\cup\cdots\cup Y_N}.
\]
Each $u^{[N]}$ belongs to $c_{00}(Y)$, and
Lemma~\ref{lem:translated-cutoff-estimates} gives
\[
u^{[N]}\longrightarrow u
\quad\text{in }\ell_{w_\nu}^{p,q}(Y).
\]
Since $T$ is bounded on $\ell_{w_\nu}^{p,q}$ by
Lemma~\ref{lem:localized-matrix-boundedness},
\[
Tu^{[N]}\longrightarrow Tu
\quad\text{in }\ell_{w_\nu}^{p,q}(X).
\]
Applying \eqref{eq:finite-support-lower-stability} to $u^{[N]}$ and
passing to the limit gives
\begin{equation}
\label{eq:cutoff-lower-stability}
\|u\|_{\ell_{w_\nu}^{p,q}(Y)}
\le
C_0\|Tu\|_{\ell_{w_\nu}^{p,q}(X)}.
\end{equation}
Using
\[
Tu
=
\Theta_{R,y_0}^XTb
+
\bigl(
T\Theta_{R,y_0}^Y-\Theta_{R,y_0}^XT
\bigr)b,
\]
the powered triangle inequality,
\eqref{eq:cutoff-lower-stability}, and
Lemma~\ref{lem:translated-commutator} give
\[
\|u\|_{\ell_{w_\nu}^{p,q}(Y)}^s
\lesssim_{p,q,\nu}
C_0^s
\|\Theta_{R,y_0}^XTb\|_{\ell_{w_\nu}^{p,q}(X)}^s
+
C_0^sR^{-s}
\|u\|_{\ell_{w_\nu}^{p,q}(Y)}^s.
\]
Choose $R=R_0$ sufficiently large, depending only on $C_0$ and the
fixed data, so that the last term can be absorbed. Then
\[
\|\Theta_{R_0,y_0}^Yb\|_{\ell_{w_\nu}^{p,q}(Y)}
\lesssim
\|\Theta_{R_0,y_0}^XTb\|_{\ell_{w_\nu}^{p,q}(X)}.
\]
By Lemma~\ref{lem:localized-matrix-boundedness},
$Tb\in\ell^\infty(X)$. Hence
Lemma~\ref{lem:translated-cutoff-estimates} gives
\[
(1+|y_0|)^\nu
\|b\|_{\ell^\infty(Y)}
\lesssim_\nu
\|\Theta_{R_0,y_0}^Yb\|_{\ell_{w_\nu}^{p,q}(Y)}
\]
and
\[
\|\Theta_{R_0,y_0}^XTb\|_{\ell_{w_\nu}^{p,q}(X)}
\lesssim
R_0^{\frac2p+|\nu|+\frac1q}
(1+|y_0|)^\nu
\|Tb\|_{\ell^\infty(X)}.
\]
The factor $(1+|y_0|)^\nu$ cancels. Since $R_0$ is fixed,
\eqref{eq:infty-lower-stability} follows.
\end{proof}

\subsection{Transfer to $\ell^2$}
\label{subsec:transfer-l2}

We now combine Proposition~\ref{prop:mixed-to-infty-stability}
with the $p$-independence theorem for the Sj\"ostrand class in
\cite[Theorem~2.1]{ShinSun2009}.

\smallskip
\noindent
Let $U,V\subset\mathbb R^d$ be relatively separated, regarded as
the row and column index sets, respectively. The \emph{Sj\"ostrand
class} $\mathcal C(U,V)$ consists of the matrices
\(
A=(A_{x,y})_{x\in U,\;y\in V}
\)
for which
\[
\|A\|_{\mathcal C(U,V)}
:=
\sum_{k\in\mathbb Z^d}
\sup_{\substack{x\in U,\;y\in V\\
x-y\in k+[0,1)^d}}
|A_{x,y}|
<\infty,
\]
where the supremum over an empty set is understood to be zero.

\smallskip
\noindent
We shall use the elementary observation that
\[
|A_{x,y}|
\lesssim
(1+|x-y|)^{-N},
\qquad N>d,
\]
implies $A\in\mathcal C(U,V)$. Indeed, if
$x-y\in k+[0,1)^d$, then
\[
1+|x-y|\asymp_d1+|k|,
\]
and hence
\[
\|A\|_{\mathcal C(U,V)}
\lesssim_{d,N}
\sum_{k\in\mathbb Z^d}(1+|k|)^{-N}
<\infty.
\]

\begin{theorem}[Localized lower-stability transfer]
\label{thm:localized-lower-stability-transfer}
Let $Y\subset\mathbb C$ be relatively separated, let
\[
X=X^{(1)}\sqcup\cdots\sqcup X^{(J)}
\]
be a finite labeled union of relatively separated subsets of
$\mathbb C$, and let $T:Y\to X$ be rapidly localized. Fix
$0<p,q\le\infty$ and $\nu\in\mathbb R$. Suppose that there exists $C_0>0$ such that
\begin{equation*}
\label{eq:mixed-lower-stability}
\|c\|_{\ell_{w_\nu}^{p,q}(Y)}
\le
C_0\|Tc\|_{\ell_{w_\nu}^{p,q}(X)},
\qquad c\in c_{00}(Y).
\end{equation*}
Then
\begin{equation}
\label{eq:hilbert-lower-stability}
\|c\|_{\ell^2(Y)}
\lesssim_{p,q,\nu,C_0}
\|Tc\|_{\ell^2(X)},
\qquad c\in\ell^2(Y).
\end{equation}
\end{theorem}

\begin{proof}
Proposition~\ref{prop:mixed-to-infty-stability} gives
\[
\|b\|_{\ell^\infty(Y)}
\lesssim_{p,q,\nu,C_0}
\|Tb\|_{\ell^\infty(X)},
\qquad b\in\ell^\infty(Y),
\]
whereas Lemma~\ref{lem:localized-matrix-boundedness}, applied with
$p=q=\infty$ and $\nu=0$, gives
\[
\|Tb\|_{\ell^\infty(X)}
\lesssim
\|b\|_{\ell^\infty(Y)}.
\]
Thus $T$ is two-sided stable from $\ell^\infty(Y)$ to
$\ell^\infty(X)$.

\smallskip
\noindent
To apply the $p$-independence theorem for the Sj\"ostrand class in
\cite[Theorem~2.1]{ShinSun2009}, realize the labels at distinct
fixed heights. Set
\[
\widehat Y:=Y\times\{0\},
\qquad
\widehat X
:=
\bigcup_{j=1}^J
\bigl(X^{(j)}\times\{j\}\bigr)
\subset\mathbb R^3.
\]
Identifying the labeled copy of $x\in X^{(j)}$ with $(x,j)$, define
the lifted matrix $\widehat T$ by
\[
\widehat T_{(x,j),(y,0)}
:=
T_{(x,j),y},
\qquad
x\in X^{(j)},\ y\in Y.
\]
Both $\widehat X$ and $\widehat Y$ are relatively separated. Since
the set of heights is finite,
\[
1+|(x,j)-(y,0)|
\asymp_J
1+|x-y|.
\]
Hence $\widehat T$ is rapidly localized in $\mathbb R^3$, and
therefore
\[
\widehat T\in\mathcal C(\widehat X,\widehat Y).
\]
Under the canonical isometric identifications
\[
\ell^\infty(\widehat Y)\cong\ell^\infty(Y),
\qquad
\ell^\infty(\widehat X)\cong\ell^\infty(X),
\]
the action of $\widehat T$ agrees with that of $T$. Thus
$\widehat T$ is two-sided stable on $\ell^\infty$.

\smallskip
\noindent
By \cite[Theorem~2.1]{ShinSun2009}, the $\ell^\infty$-stability of
$\widehat T$ implies its $\ell^2$-stability. Hence
\[
\|c\|_{\ell^2(\widehat Y)}
\lesssim
\|\widehat Tc\|_{\ell^2(\widehat X)}.
\]
Under the canonical isometric identifications
\[
\ell^2(\widehat Y)\cong\ell^2(Y),
\qquad
\ell^2(\widehat X)\cong\ell^2(X),
\]
this is precisely \eqref{eq:hilbert-lower-stability}.
\end{proof}

\smallskip
\noindent
The finite-layer lift is introduced only at the unweighted
$\ell^\infty$-stage; no annular mixed-norm space is transported to
the lifted configuration.

\section{Mixed-to-Hilbert transfers}
\label{sec:mixed-hilbert-transfer}

We now combine the localized stability transfer from Section~\ref{sec:localized-stability}
with the Fock-space tools developed in Sections~\ref{sec:structural-inputs} and \ref{sec:local-geometry}.
Both necessity mechanisms are implemented through one fixed Hilbert
sampling lattice.

\subsection{A fixed Hilbert sampling lattice}
\label{subsec:fixed-hilbert-lattice}

Choose
\(
0<\delta_0<\sqrt{\frac{\pi}{\alpha}}
\)
and set
\[
\Omega:=\delta_0(\mathbb Z+i\mathbb Z).
\]
Then
\[
D^-(\Omega)=\delta_0^{-2}>\frac{\alpha}{\pi},
\]
so Theorem~\ref{thm:classical-hilbert-density} shows that $\Omega$ is sampling
for $\mathcal F_\alpha^2$.

\smallskip
\noindent
Let $(\varphi_\omega)_{\omega\in\Omega}$ be the canonical Hilbert
dual family from Proposition~\ref{prop:localized-hilbert-atoms}, and
define
\[
(C_\Omega f)_\omega
:=
f(\omega)e^{-\frac{\alpha}{2}|\omega|^2},
\qquad
S_\Omega^\varphi c
:=
\sum_{\omega\in\Omega}c_\omega\varphi_\omega .
\]

\begin{proposition}
\label{prop:fixed-hilbert-lattice}
Let $0<p,q\le\infty$.
The lattice $\Omega$ has the following properties.

\begin{enumerate}[label=\textnormal{(\alph*)}]
\item
There exist constants $A_\Omega,B_\Omega>0$ such that
\[
A_\Omega\|f\|_{\mathcal F_\alpha^2}^2
\le
\|C_\Omega f\|_{\ell^2(\Omega)}^2
\le
B_\Omega\|f\|_{\mathcal F_\alpha^2}^2,
\qquad
f\in\mathcal F_\alpha^2.
\]

\item
The operators
\[
C_\Omega:
\mathcal F_\alpha^{p,q}
\longrightarrow
\ell_{\textnormal{nat}}^{p,q}(\Omega),
\qquad
S_\Omega^\varphi:
\ell_{\textnormal{nat}}^{p,q}(\Omega)
\longrightarrow
\mathcal F_\alpha^{p,q}
\]
are bounded, and
\(
S_\Omega^\varphi C_\Omega=I\) on \(\mathcal F_\alpha^{p,q}\).

\item
The coefficient operator
\(
P_\Omega:=C_\Omega S_\Omega^\varphi
\)
is a bounded projection on $\ell_{\textnormal{nat}}^{p,q}(\Omega)$ and satisfies
\[
P_\Omega^2=P_\Omega,
\qquad
S_\Omega^\varphi P_\Omega=S_\Omega^\varphi,
\qquad
P_\Omega C_\Omega=C_\Omega.
\]

\item[\textnormal{(d)}] The matrices of $P_\Omega$ and $I-P_\Omega$ are rapidly
localized.
\end{enumerate}
\end{proposition}

\begin{proof}
Part~(a) is the Hilbert frame inequality for the sampling set $\Omega$.

\smallskip
\noindent
Since $\Omega$ is separated, Proposition~\ref{prop:upper-restriction}
and the point--shell weight comparison give
\[
\|C_\Omega f\|_{\ell_{\textnormal{nat}}^{p,q}(\Omega)}
\asymp_{p,q}
\|R_{\alpha,\Omega}^{p,q}f\|_{\ell^{p,q}(\Omega)}
\lesssim_{\alpha,p,q}
\|f\|_{\mathcal F_\alpha^{p,q}}.
\]
The family $(\varphi_\omega)_{\omega\in\Omega}$ is rapidly Fock
localized by Proposition~\ref{prop:localized-hilbert-atoms}.
Hence Theorem~\ref{thm:weighted-localized-synthesis}, with
$\sigma=0$, gives
\[
\|S_\Omega^\varphi c\|_{\mathcal F_\alpha^{p,q}}
\lesssim
\|c\|_{\ell_{\mathrm{nat}}^{p,q}(\Omega)}.
\]
The identity
\[
S_\Omega^\varphi C_\Omega=I
\]
is Proposition~\ref{prop:hilbert-reconstruction-transfer} applied to
the Hilbert sampling set $\Omega$.

\smallskip
\noindent
It follows that $P_\Omega=C_\Omega S_\Omega^\varphi$ is bounded and
\[
P_\Omega^2
=
C_\Omega S_\Omega^\varphi C_\Omega S_\Omega^\varphi
=
P_\Omega.
\]
The remaining identities in (c) follow similarly.

\smallskip
\noindent
Finally,
\[
(P_\Omega)_{\omega,\eta}
=
\varphi_\eta(\omega)e^{-\frac{\alpha}{2}|\omega|^2},
\qquad
\omega,\eta\in\Omega.
\]
By the rapid localization of the dual atoms,
\[
|(P_\Omega)_{\omega,\eta}|
\lesssim_N
(1+|\omega-\eta|)^{-N}
\]
for every $N>0$. Thus $P_\Omega$ is rapidly localized, and so is
$I-P_\Omega$.
\end{proof}

\subsection{Sampling transfer}
\label{subsec:sampling-transfer}

Let $\Gamma\subset\mathbb C$ be separated and define
\[
A_{\Gamma,\Omega}
:=
C_\Gamma S_\Omega^\varphi .
\]
Thus
\[
(A_{\Gamma,\Omega}c)_\gamma
=
(S_\Omega^\varphi c)(\gamma)
e^{-\frac{\alpha}{2}|\gamma|^2},
\qquad
\gamma\in\Gamma.
\]
Hence $A_{\Gamma,\Omega}$ maps coefficient sequences on $\Omega$
to their Gaussian-normalized samples on $\Gamma$.

\smallskip
\noindent
Its matrix entries are
\[
(A_{\Gamma,\Omega})_{\gamma,\omega}
=
\varphi_\omega(\gamma)e^{-\frac{\alpha}{2}|\gamma|^2}.
\]
By Proposition~\ref{prop:localized-hilbert-atoms}(i), applied to the
fixed Hilbert sampling lattice $\Omega$,
\[
|(A_{\Gamma,\Omega})_{\gamma,\omega}|
\lesssim_N
(1+|\gamma-\omega|)^{-N},
\]
so $A_{\Gamma,\Omega}$ is rapidly localized. 

\smallskip
\noindent
Define the augmented operator
\[
T_{\Gamma,\Omega}^{\mathrm{aug}}c
:=
\bigl(A_{\Gamma,\Omega}c,(I-P_\Omega)c\bigr),
\]
viewed as a matrix
\[
T_{\Gamma,\Omega}^{\mathrm{aug}}
:
\Omega\longrightarrow \Gamma\sqcup\Omega,
\]
where the target is the labeled union introduced in Section~\ref{sec:localized-stability}.
Since $A_{\Gamma,\Omega}$ and $I-P_\Omega$ are rapidly localized,
so is $T_{\Gamma,\Omega}^{\mathrm{aug}}$.

\begin{theorem}[Sampling transfer to the Hilbert scale]
\label{thm:sampling-hilbert-transfer}
Let $0<p,q\le\infty$, and let $\Gamma\subset\mathbb C$ be separated.
Suppose that there exists $C_0>0$ such that
\[
\|f\|_{\mathcal F_\alpha^{p,q}}
\le C_0
\|R_{\alpha,\Gamma}^{p,q}f\|_{\ell^{p,q}(\Gamma)},
\qquad
f\in\mathcal F_\alpha^{p,q}.
\]
Then
\[
\|f\|_{\mathcal F_\alpha^2}
\lesssim_{\alpha,p,q,C_0}
\|C_\Gamma f\|_{\ell^2(\Gamma)},
\qquad
f\in\mathcal F_\alpha^2.
\]
\end{theorem}

\begin{proof}
Put
\[
s:=\min\{1,p,q\}.
\]
We first show that
\[
\|c\|_{\ell_{\textnormal{nat}}^{p,q}(\Omega)}
\lesssim_{\alpha,p,q,C_0}
\|T_{\Gamma,\Omega}^{\mathrm{aug}}c\|
_{\ell_{\textnormal{nat}}^{p,q}(\Gamma\sqcup\Omega)},
\qquad
c\in c_{00}(\Omega).
\]
Since
\[
c=P_\Omega c+(I-P_\Omega)c,
\]
the powered triangle inequality gives
\[
\|c\|_{\ell_{\textnormal{nat}}^{p,q}(\Omega)}^s
\le
\|P_\Omega c\|_{\ell_{\textnormal{nat}}^{p,q}(\Omega)}^s
+
\|(I-P_\Omega)c\|_{\ell_{\textnormal{nat}}^{p,q}(\Omega)}^s.
\]
By the boundedness of $C_\Omega$,
\[
\|P_\Omega c\|_{\ell_{\textnormal{nat}}^{p,q}(\Omega)}
=
\|C_\Omega S_\Omega^\varphi c\|_{\ell_{\textnormal{nat}}^{p,q}(\Omega)}
\lesssim_{\alpha,p,q}
\|S_\Omega^\varphi c\|_{\mathcal F_\alpha^{p,q}}.
\]
The assumed lower-sampling estimate and the point--shell comparison
then give
\[
\|S_\Omega^\varphi c\|_{\mathcal F_\alpha^{p,q}}
\le C_0
\|R_{\alpha,\Gamma}^{p,q}
(S_\Omega^\varphi c)\|_{\ell^{p,q}(\Gamma)}
\asymp_{p,q}
C_0\|A_{\Gamma,\Omega}c\|_{\ell_{\textnormal{nat}}^{p,q}(\Gamma)}.
\]
Hence
\[
\|c\|_{\ell_{\textnormal{nat}}^{p,q}(\Omega)}^s
\lesssim_{\alpha,p,q}
C_0^s\|A_{\Gamma,\Omega}c\|_{\ell_{\textnormal{nat}}^{p,q}(\Gamma)}^s
+
\|(I-P_\Omega)c\|_{\ell_{\textnormal{nat}}^{p,q}(\Omega)}^s.
\]
Since both component norms are dominated by the mixed norm on the
labeled union $\Gamma\sqcup\Omega$, we obtain
\[
\|c\|_{\ell_{\textnormal{nat}}^{p,q}(\Omega)}
\lesssim_{\alpha,p,q,C_0}
\|T_{\Gamma,\Omega}^{\mathrm{aug}}c\|
_{\ell_{\textnormal{nat}}^{p,q}(\Gamma\sqcup\Omega)}.
\]

\smallskip
\noindent
Since $T_{\Gamma,\Omega}^{\mathrm{aug}}$ is rapidly localized,
Theorem~\ref{thm:localized-lower-stability-transfer} gives
\[
\|c\|_{\ell^2(\Omega)}
\lesssim_{\alpha,p,q,C_0}
\|T_{\Gamma,\Omega}^{\mathrm{aug}}c\|
_{\ell^2(\Gamma\sqcup\Omega)},
\qquad
c\in\ell^2(\Omega).
\]
Now let $f\in\mathcal F_\alpha^2$ and take $c=C_\Omega f$.
Then
\(P_\Omega c=c\) and \(
S_\Omega^\varphi c=f\),
so
\[
(I-P_\Omega)c=0,
\qquad
A_{\Gamma,\Omega}c=C_\Gamma f.
\]
Therefore
\[
\|C_\Omega f\|_{\ell^2(\Omega)}
\lesssim
\|C_\Gamma f\|_{\ell^2(\Gamma)}.
\]
The lower Hilbert frame bound for $\Omega$ now gives
\[
\|f\|_{\mathcal F_\alpha^2}
\lesssim
\|C_\Gamma f\|_{\ell^2(\Gamma)}.
\]
\end{proof}

\subsection{Interpolation transfer}
\label{subsec:interpolation-transfer}

Let $\Lambda\subset\mathbb C$ be separated. Define the Gaussian
cross-Gram matrix
\[
M_{\Omega,\Lambda}d
:=
C_\Omega
\left(
\sum_{\lambda\in\Lambda}
d_\lambda\kappa_{\alpha,\lambda}
\right),
\qquad
d\in c_{00}(\Lambda).
\]
Its entries are
\[
(M_{\Omega,\Lambda})_{\omega,\lambda}
=
\kappa_{\alpha,\lambda}(\omega)
e^{-\frac{\alpha}{2}|\omega|^2},
\qquad
\omega\in\Omega,\ \lambda\in\Lambda,
\]
and hence
\[
|(M_{\Omega,\Lambda})_{\omega,\lambda}|
=
e^{-\frac{\alpha}{2}|\omega-\lambda|^2}.
\]
Therefore $M_{\Omega,\Lambda}:\Lambda\to\Omega$ is rapidly localized.

\begin{theorem}[Weighted synthesis transfer]
\label{thm:weighted-synthesis-transfer}
Let $0<p,q\le\infty$ and let $\Lambda\subset\mathbb C$ be separated. Suppose that there exists $C_0>0$ such that
\begin{equation}
\label{eq:weighted-lower-synthesis}
\|d\|_{\ell_{\textnormal{nat}; \sigma(p,q)}^{p^*,q^*}(\Lambda)}
\le C_0
\left\|
\sum_{\lambda\in\Lambda}
d_\lambda\kappa_{\alpha,\lambda}
\right\|_{\mathcal F_{\alpha;\sigma(p,q)}^{p^*,q^*}},
\qquad
d\in c_{00}(\Lambda).
\end{equation}
Then
\[
\|d\|_{\ell^2(\Lambda)}
\lesssim_{\alpha,p,q,C_0}
\left\|
\sum_{\lambda\in\Lambda}
d_\lambda\kappa_{\alpha,\lambda}
\right\|_{\mathcal F_\alpha^2},
\qquad
d\in c_{00}(\Lambda).
\]
\end{theorem}

\begin{proof}
For simplicity, throughout this proof, set
\[
\nu:=-\nu(p,q)
=\frac1p-\frac1q.
\]
By \eqref{eq:duality-correction}, for $Z=\Lambda,\Omega$,
\[
\ell_{\mathrm{nat};\sigma(p,q)}^{p^*,q^*}(Z)
=
\ell_{w_\nu}^{p^*,q^*}(Z).
\]
The coefficient weight in
Theorem~\ref{thm:weighted-localized-synthesis}, applied with
$p^*,q^*$ and $\sigma=\sigma(p,q)$, is precisely $\nu$.

\smallskip
\noindent
Fix $d\in c_{00}(\Lambda)$. Since
\[
\sum_{\lambda\in\Lambda}
d_\lambda\kappa_{\alpha,\lambda}
\in\mathcal F_\alpha^2,
\]
the reconstruction formula for the fixed Hilbert sampling lattice
$\Omega$ gives
\[
\sum_{\lambda\in\Lambda}
d_\lambda\kappa_{\alpha,\lambda}
=
S_\Omega^\varphi
M_{\Omega,\Lambda}d.
\]
Since $d$ is finitely supported and the columns of
$M_{\Omega,\Lambda}$ have Gaussian decay,
\[
M_{\Omega,\Lambda}d\in
\ell_{w_\nu}^{p^*,q^*}(\Omega).
\]

\smallskip
\noindent
Since $(\varphi_\omega)_{\omega\in\Omega}$ is rapidly Fock localized,
Theorem~\ref{thm:weighted-localized-synthesis} therefore gives
\[
\left\|
\sum_{\lambda\in\Lambda}
d_\lambda\kappa_{\alpha,\lambda}
\right\|_{\mathcal F_{\alpha;\sigma(p,q)}^{p^*,q^*}}
\lesssim_{\alpha,p,q}
\|M_{\Omega,\Lambda}d\|_{\ell_{w_\nu}^{p^*,q^*}(\Omega)}.
\]
Combining this with \eqref{eq:weighted-lower-synthesis}, we obtain
\[
\|d\|_{\ell_{w_\nu}^{p^*,q^*}(\Lambda)}
\le
C_0
\left\|
\sum_{\lambda\in\Lambda}
d_\lambda\kappa_{\alpha,\lambda}
\right\|_{\mathcal F_{\alpha;\sigma(p,q)}^{p^*,q^*}}
\lesssim_{\alpha,p,q}
C_0
\|M_{\Omega,\Lambda}d\|_{\ell_{w_\nu}^{p^*,q^*}(\Omega)}.
\]
The matrix $M_{\Omega,\Lambda}$ is rapidly localized, so
Theorem~\ref{thm:localized-lower-stability-transfer} yields
\[
\|d\|_{\ell^2(\Lambda)}
\lesssim_{\alpha,p,q,C_0}
\|M_{\Omega,\Lambda}d\|_{\ell^2(\Omega)}.
\]
Finally, by the upper Hilbert frame bound for $\Omega$,
\[
\|M_{\Omega,\Lambda}d\|_{\ell^2(\Omega)}
=
\left\|
C_\Omega
\left(
\sum_{\lambda\in\Lambda}
d_\lambda\kappa_{\alpha,\lambda}
\right)
\right\|_{\ell^2(\Omega)}
\lesssim
\left\|
\sum_{\lambda\in\Lambda}
d_\lambda\kappa_{\alpha,\lambda}
\right\|_{\mathcal F_\alpha^2}.
\]
The conclusion follows.
\end{proof}

\section{Necessity}
\label{sec:necessity}

We now prove the necessity directions of the main theorems.
Sampling is reduced to the Hilbert setting through a separated
carrier and Theorem~\ref{thm:sampling-hilbert-transfer}, while
interpolation is reduced through finite-support norming and
Theorem~\ref{thm:weighted-synthesis-transfer}. The classical
Hilbert density theorem then gives the strict density conditions.

\subsection{Sampling necessity}
\label{subsec:sampling-necessity}

\begin{theorem}
\label{thm:sampling-necessity}
Let $0<p,q\le\infty$, and let $\Lambda\subset\mathbb C$ be locally
finite. If $\Lambda$ is sampling for $\mathcal F_\alpha^{p,q}$, then
\[
D^-_{\mathrm{sep}}(\Lambda)>\frac{\alpha}{\pi}.
\]
If $0<p<\infty$, then $\Lambda$ is also relatively separated.
\end{theorem}

\begin{proof}
Suppose that $\Lambda$ is sampling for $\mathcal F_\alpha^{p,q}$.
If $0<p<\infty$, the upper sampling inequality and
Proposition~\ref{prop:local-geometric-consequences} imply that $\Lambda$
is relatively separated.

\smallskip
\noindent
By Proposition~\ref{prop:separated-sampling-carrier}, there exist
$0<\delta<1$ and a $\delta$-separated subset
$\Gamma_\delta\subset\Lambda$ such that
\[
\|f\|_{\mathcal F_\alpha^{p,q}}
\lesssim_{\alpha,p,q}
\|R_{\alpha,\Gamma_\delta}^{p,q}f\|_{\ell^{p,q}(\Gamma_\delta)},
\qquad
f\in\mathcal F_\alpha^{p,q}.
\]
Theorem~\ref{thm:sampling-hilbert-transfer} therefore gives
\[
\|f\|_{\mathcal F_\alpha^2}
\lesssim
\|C_{\Gamma_\delta}f\|_{\ell^2(\Gamma_\delta)},
\qquad
f\in\mathcal F_\alpha^2.
\]
Since $\Gamma_\delta$ is separated, hence relatively separated,
Proposition~\ref{prop:upper-restriction}, with $p=q=2$, gives
\[
\|C_{\Gamma_\delta}f\|_{\ell^2(\Gamma_\delta)}
\lesssim
\|f\|_{\mathcal F_\alpha^2}.
\]
Combined with the preceding lower estimate, this shows that
$\Gamma_\delta$ is sampling for $\mathcal F_\alpha^2$.
By Theorem~\ref{thm:classical-hilbert-density},
\[
D^-(\Gamma_\delta)>\frac{\alpha}{\pi}.
\]
Since $\Gamma_\delta\subset\Lambda$ is separated,
\[
D^-_{\mathrm{sep}}(\Lambda)
\ge D^-(\Gamma_\delta)
>\frac{\alpha}{\pi}.
\]
\end{proof}

\subsection{Interpolation necessity}
\label{subsec:interpolation-necessity}

We now prove interpolation necessity for both
$\mathcal F_\alpha^{p,q}$ and the little endpoint
$f_\alpha^{p,\infty}$. We use the controlled-preimage consequence
of interpolation recorded in Remark~\ref{rem:uniform-interpolation}.
The following finite-support norming identity replaces full duality
of the mixed sequence spaces in the argument.

\begin{lemma}
\label{lem:finite-support-norming}
Let $\Lambda\subset\mathbb C$ be locally finite and
$0<p,q\le\infty$.  For every $b\in c_{00}(\Lambda)$,
\begin{equation*}
\label{eq:finite-support-norming}
\sup_{\substack{a\in c_{00}(\Lambda)\\
\|a\|_{\ell^{p,q}(\Lambda)}\le1}}
\left|
\sum_{\lambda\in\Lambda}
a_\lambda\overline{b_\lambda}
\right|
=
\|b\|_{\ell^{p^*,q^*}(\Lambda)}.
\end{equation*}
\end{lemma}

\begin{proof}
For a finite set $E$, the finite-dimensional norming identity gives
\[
\sup_{\|x\|_{\ell^p(E)}\le1}
\left|
\sum_{j\in E}x_j\overline{y_j}
\right|
=
\|y\|_{\ell^{p^*}(E)}.
\]
For $1<p\le\infty$ this is the usual finite-dimensional duality, while
for $0<p\le1$ the supremum is $\|y\|_{\ell^\infty(E)}$.

\smallskip
\noindent
Set
\[
u_k:=\|a|_{\Lambda_k}\|_{\ell^p},
\qquad
v_k:=\|b|_{\Lambda_k}\|_{\ell^{p^*}}.
\]
Applying the preceding identity on each annular block gives
\[
\left|
\sum_{\lambda\in\Lambda}
a_\lambda\overline{b_\lambda}
\right|
\le
\sum_{k\ge0}
\left|
\sum_{\lambda\in\Lambda_k}
a_\lambda\overline{b_\lambda}
\right|
\le
\sum_{k\ge0}u_kv_k.
\]
Since both sequences are finitely supported, applying the corresponding
finite-dimensional norming identity with exponent $q$ yields
\[
\sup_{\substack{a\in c_{00}(\Lambda)\\
\|a\|_{\ell^{p,q}(\Lambda)}\le1}}
\left|
\sum_{\lambda\in\Lambda}
a_\lambda\overline{b_\lambda}
\right|
\le
\|b\|_{\ell^{p^*,q^*}(\Lambda)}.
\]
Conversely, choose a finitely supported nonnegative sequence $(c_k)_{k \geq 0}$ with
\[
\|(c_k)_{k \geq 0}\|_{\ell^q}\le1,
\qquad
\sum_{k\ge0}c_kv_k
=
\|(v_k)_{k \geq 0}\|_{\ell^{q^*}}.
\]
For each $k$ with $v_k>0$, choose
$a|_{\Lambda_k}$ with
\(
\|a|_{\Lambda_k}\|_{\ell^p}=c_k
\)
that norms $b|_{\Lambda_k}$, with the block phases aligned. Then
\[
\|a\|_{\ell^{p,q}(\Lambda)}\le1 \qquad \text{and} \qquad
\left|
\sum_{\lambda\in\Lambda}
a_\lambda\overline{b_\lambda}
\right|
=
\sum_{k\ge0}c_kv_k
=
\|b\|_{\ell^{p^*,q^*}(\Lambda)}.
\]
\end{proof}

\begin{theorem}
\label{thm:interpolation-necessity}
Let $\Lambda\subset\mathbb C$ be locally finite.

\begin{enumerate}[label=\textnormal{(\alph*)}]
\item
Let $0<p,q\le\infty$. If $\Lambda$ is interpolating for
$\mathcal F_\alpha^{p,q}$, then $\Lambda$ is separated and
\(
D^+(\Lambda)<\frac{\alpha}{\pi}.
\)

\smallskip
\noindent

\item
Let $0<p\le\infty$. If $\Lambda$ is interpolating for
$f_\alpha^{p,\infty}$, then $\Lambda$ is separated and
\(
D^+(\Lambda)<\frac{\alpha}{\pi}.
\)
\end{enumerate}
\end{theorem}

\begin{proof}
Under either hypothesis,
Proposition~\ref{prop:local-geometric-consequences} shows that $\Lambda$
is separated.

\smallskip
\noindent
For (a), since $\Lambda$ is interpolating for $\mathcal F_\alpha^{p,q}$, the quasi-Banach open mapping theorem argument in Remark~\ref{rem:uniform-interpolation} shows that, for
every $a\in c_{00}(\Lambda) \subset \ell^{p,q}(\Lambda)$, one may choose
$f_a\in\mathcal F_\alpha^{p,q}$ such that
\[
R_{\alpha,\Lambda}^{p,q}f_a=a,
\qquad
\|f_a\|_{\mathcal F_\alpha^{p,q}}
\lesssim_{\alpha,p,q}
\|a\|_{\ell^{p,q}(\Lambda)}.
\]
Fix $b\in c_{00}(\Lambda)$ and set
\[
g_b
:=
\sum_{\lambda\in\Lambda}
b_\lambda(1+|\lambda|)^{\nu(p,q)}
\kappa_{\alpha,\lambda}.
\]
If $\|a\|_{\ell^{p,q}(\Lambda)}\le1$, then
\[
a_\lambda
=
f_a(\lambda)e^{-\frac{\alpha}{2}|\lambda|^2}
(1+|\lambda|)^{\nu(p,q)}.
\]
Hence the reproducing identity and the forward Gaussian pairing (F5) give
\[
\left|
\sum_{\lambda\in\Lambda}
a_\lambda\overline{b_\lambda}
\right|
=
|\langle f_a,g_b\rangle_\alpha|
\lesssim
\|g_b\|_{\mathcal F_{\alpha;\sigma(p,q)}^{p^*,q^*}}.
\]
Taking the supremum over such $a$ and applying
Lemma~\ref{lem:finite-support-norming}, we obtain
\[
\|b\|_{\ell^{p^*,q^*}(\Lambda)}
\lesssim
\|g_b\|_{\mathcal F_{\alpha;\sigma(p,q)}^{p^*,q^*}}.
\]
Now set
\[
d_\lambda
:=
b_\lambda(1+|\lambda|)^{\nu(p,q)}.
\]
By the point--shell comparison and \eqref{eq:duality-correction},
\[
\|b\|_{\ell^{p^*,q^*}(\Lambda)}
\asymp_{p,q}
\|d\|_{\ell_{\mathrm{nat};\sigma(p,q)}^{p^*,q^*}(\Lambda)},
\]
and therefore
\[
\|d\|_{\ell_{\mathrm{nat};\sigma(p,q)}^{p^*,q^*}(\Lambda)}
\lesssim
\left\|
\sum_{\lambda\in\Lambda}
d_\lambda\kappa_{\alpha,\lambda}
\right\|_{\mathcal F_{\alpha;\sigma(p,q)}^{p^*,q^*}}.
\]
Theorem~\ref{thm:weighted-synthesis-transfer} now yields
\[
\|d\|_{\ell^2(\Lambda)}
\lesssim
\left\|
\sum_{\lambda\in\Lambda}
d_\lambda\kappa_{\alpha,\lambda}
\right\|_{\mathcal F_\alpha^2},
\qquad
d\in c_{00}(\Lambda).
\]

\smallskip
\noindent
For (b), since $\Lambda$ is interpolating for
$f_\alpha^{p,\infty}$, the quasi-Banach open mapping theorem
argument in Remark~\ref{rem:uniform-interpolation} shows that, for
every
\(
a\in c_{00}(\Lambda)\subset\ell_0^{p,\infty}(\Lambda),
\)
one may choose $f_a\in f_\alpha^{p,\infty}$ such that
\[
R_{\alpha,\Lambda}^{p,\infty}f_a=a,
\qquad
\|f_a\|_{\mathcal F_\alpha^{p,\infty}}
\lesssim_{\alpha,p}
\|a\|_{\ell^{p,\infty}(\Lambda)}.
\]
Repeating the preceding finite-support norming argument with
$q=\infty$ and
\[
d_\lambda:=b_\lambda(1+|\lambda|)^{\nu(p,\infty)},
\]
gives
\[
\|d\|_{\ell_{\textnormal{nat}; \sigma(p, \infty)}^{p^*,1}(\Lambda)}
\lesssim
\left\|
\sum_{\lambda\in\Lambda}
d_\lambda\kappa_{\alpha,\lambda}
\right\|_{\mathcal F_{\alpha;\sigma(p,\infty)}^{p^*,1}},
\qquad
d\in c_{00}(\Lambda).
\]
Applying Theorem~\ref{thm:weighted-synthesis-transfer} with
$q=\infty$ gives the same Hilbert lower-synthesis estimate.

\smallskip
\noindent
Thus, in either case,
\[
\|d\|_{\ell^2(\Lambda)}
\lesssim
\left\|
\sum_{\lambda\in\Lambda}
d_\lambda\kappa_{\alpha,\lambda}
\right\|_{\mathcal F_\alpha^2},
\qquad
d\in c_{00}(\Lambda).
\]
Since $\Lambda$ is separated, the Hilbert Bessel estimate recalled
in Subsection~\ref{SS:HilbertFock} gives
\[
\left\|
\sum_{\lambda\in\Lambda}
d_\lambda\kappa_{\alpha,\lambda}
\right\|_{\mathcal F_\alpha^2}
\lesssim
\|d\|_{\ell^2(\Lambda)},
\qquad
d\in c_{00}(\Lambda).
\]
Together with the preceding lower synthesis estimate, this shows that
$(\kappa_{\alpha,\lambda})_{\lambda\in\Lambda}$ is a Riesz sequence
in $\mathcal F_\alpha^2$. By the Hilbert interpolation--Riesz
sequence equivalence recalled in the same subsection, $\Lambda$ is
interpolating for $\mathcal F_\alpha^2$. Hence
Theorem~\ref{thm:classical-hilbert-density} gives
\[
D^+(\Lambda)<\frac{\alpha}{\pi}.
\]
\end{proof}

\section*{Acknowledgements}

X.~F. is
supported by a grant from the National Science and Technology Council, Taiwan
(NSTC 114-2115-M-A49-003-MY3).

\bigskip
\textit{AI Statement.}
The authors used artificial-intelligence tools for language editing,
\LaTeX\ formatting, and limited assistance with local mathematical
reasoning.  The proof strategy and mathematical development are the
authors' own; all mathematical content was independently written and checked by the authors, who take full responsibility for it.


\begin{thebibliography}{99}

\bibitem{BerndtssonOrtega1995}
B. Berndtsson and J. Ortega-Cerd\`a,
\emph{On interpolation and sampling in Hilbert spaces of analytic functions},
J. Reine Angew. Math. 464 (1995), 109--128.

\bibitem{BG24}
O.~Blasco and A.~Galbis,
\emph{Boundedness and compactness of Hausdorff operators on Fock spaces},
Trans. Amer. Math. Soc. \textbf{377} (2024), 5165--5196.

\bibitem{BorichevDhuezKellay2007}
A. Borichev, R. Dhuez, and K. Kellay,
\emph{Sampling and interpolation in large Bergman and Fock spaces},
J. Funct. Anal. 242 (2007), no. 2, 563--606.

\bibitem{CP16}
O.~Constantin and J.A.~Pel\'aez,
\emph{Integral operators, embedding theorems and a Littlewood--Paley formula on weighted Fock spaces},
J. Geom. Anal. \textbf{26} (2016), no.~2, 1109--1154.

\bibitem{FT23}
X.~Fang and P.T.~Tien,
\emph{Two problems on random analytic functions in Fock spaces},
Canad. J. Math. \textbf{75} (2023), no.~4, 1176--1198.

\bibitem{FT26}
X.~Fang and P.T.~Tien,
\emph{Sharp Gaussian mixed-norm theory for Fock spaces},
submitted manuscript, 2026.

\bibitem{GrochenigHaimiOrtegaRomero2019}
K. Gr\"ochenig, A. Haimi, J. Ortega-Cerd\`a, and J. L. Romero,
\emph{Strict density inequalities for sampling and interpolation in weighted spaces of holomorphic functions},
J. Funct. Anal. 277 (2019), no. 12, 108282.

\bibitem{Lindholm2001}
N. Lindholm,
\emph{Sampling in weighted $L^p$ spaces of entire functions in $\mathbb C^n$
and estimates of the Bergman kernel},
J. Funct. Anal. 182 (2001), no. 2, 390--426.

\bibitem{Liu2024}
Y.~Liu,
\emph{Fock projections on mixed norm spaces},
Mediterr. J. Math. \textbf{21} (2024), no.~6, Paper No.~171.

\bibitem{Liu2025}
Y.~Liu,
\emph{Hausdorff operators on weighted mixed norm Fock spaces},
Bull. Malays. Math. Sci. Soc. \textbf{48} (2025), no.~2, Paper No.~48.

\bibitem{LiuShiWei2027}
Y. Liu, H. Shi, and S. Wei,
\emph{Boundedness of Fock projections between mixed norm spaces},
J. Math. Anal. Appl. 565 (2027), no. 1, Paper No. 130921.

\bibitem{Lyubarskii1992}
Y.~I. Lyubarskii,
\emph{Frames in the Bargmann space of entire functions},
in \emph{Entire and Subharmonic Functions}, Adv. Soviet Math., vol.~11,
Amer. Math. Soc., Providence, RI, 1992, pp.~167--180.

\bibitem{MarcoMassanedaOrtega2003}
N.~Marco, X.~Massaneda, and J.~Ortega-Cerd\`a,
\emph{Interpolating and sampling sequences for entire functions},
Geom. Funct. Anal. \textbf{13} (2003), no.~4, 862--914.

\bibitem{NguyenLuecking2018}
P.~K. Nguyen and D.~H. Luecking,
\emph{Interpolation and sampling sequences for mixed-norm spaces},
arXiv:1801.07761, 2018.

\bibitem{OrtegaCerdaSeip1998}
J. Ortega-Cerd\`a and K. Seip,
\emph{Beurling-type density theorems for weighted $L^p$ spaces of entire functions},
J. Anal. Math. 75 (1998), 247--266.

\bibitem{Seip1992}
K.~Seip,
\emph{Density theorems for sampling and interpolation in the
Bargmann--Fock space. I},
J. Reine Angew. Math. \textbf{429} (1992), 91--106.

\bibitem{SeipWallsten1992}
K.~Seip and R.~Wallst\'en,
\emph{Density theorems for sampling and interpolation in the
Bargmann--Fock space. II},
J. Reine Angew. Math. \textbf{429} (1992), 107--113.

\bibitem{ShinSun2009}
C.~E. Shin and Q.~Sun,
\emph{Stability of localized operators},
J. Funct. Anal. \textbf{256} (2009), no.~8, 2417--2439.

\bibitem{SignahlToft2012}
M.~Signahl and J.~Toft,
\emph{Mapping properties for the Bargmann transform on modulation spaces},
J. Pseudo-Differ. Oper. Appl. \textbf{3} (2012), no.~1, 1--30.

\bibitem{Sun2007}
Q.~Sun,
\emph{Wiener's lemma for infinite matrices},
Trans. Amer. Math. Soc. \textbf{359} (2007), no.~7, 3099--3123.

\bibitem{Zhu2012}
K.~Zhu,
\emph{Analysis on Fock Spaces},
Graduate Texts in Mathematics, vol.~263, Springer, New York, 2012.

\end{thebibliography}
\end{document}